\documentclass[a4]{article}
\usepackage{CJKutf8}
\usepackage[utf8]{inputenc}
\usepackage[T1]{fontenc}
\usepackage{lmodern}
\usepackage{microtype}
\usepackage{amsmath}
\usepackage{amssymb}
\usepackage{amsthm}
\usepackage{mathtools}
\usepackage{bm}
\usepackage{cases}
\usepackage{cite}
\usepackage{graphicx} % draft: ignore figures for faster compilation
\usepackage{tikz}
\usepackage[
  left=25mm,
  right=25mm,
  top=25mm,bottom=25mm,
  marginparwidth=30mm,
  marginparsep=3mm
]{geometry}
\usepackage{setspace}
\usepackage{accents}
\usepackage{comment}
\usepackage{algorithm}
\usepackage{algpseudocode}
\usepackage{subcaption}
\usepackage{url}
\usepackage[
  colorlinks=true,
  linkcolor=blue,
  citecolor=green!50!black,
  urlcolor=cyan,
  bookmarks=true,
  bookmarksnumbered=true,
  pdfusetitle
]{hyperref}
\usepackage{multirow}
\usepackage{colortbl}
\usepackage{xcolor}
\usepackage[normalem]{ulem}  % For latexdiff-style \sout{}

\usepackage{pgfplots}
\pgfplotsset{compat=1.18}

\newtheorem{theorem}{Theorem}[section]
\newtheorem{lemma}[theorem]{Lemma}
\newtheorem{corollary}[theorem]{Corollary}

\newtheorem{definition}[theorem]{Definition}
\newtheorem{construction}[theorem]{Construction}
\newtheorem{problem}[theorem]{Problem}
\newtheorem{remark}[theorem]{Remark}
\newtheorem{assumption}[theorem]{Assumption}

\title{A Green's Function–Based Enclosure Framework for Poisson's Equation and Generalized Sub- and Super-Solutions}

\newcommand{\fundsol}{\Gamma}
\newcommand{\belem}{\gamma}
\newcommand{\harmonic}{H_{\intpt}}
\newcommand{\intpt}{s_{\rm int}}
\newcommand{\intcoef}{a_{\rm int}}
\newcommand{\testFuncMap}{\Phi}
\newcommand{\testFuncSet}{\mathcal{T}_{\intpt}}

\newcommand{\del}[1]{}

\newif\ifshowmemo
\showmemofalse

\ifshowmemo
    \usepackage[color=orange!10]{todonotes}
    \setuptodonotes{inline}
    \let\oldtodo\todo
    \renewcommand{\todo}[1]{\oldtodo{\begin{CJK}{UTF8}{maru} #1 \end{CJK}}}
\else
    \usepackage[disable]{todonotes}
\fi

\author{
Kazuaki Tanaka\textsuperscript{1,*} \and
Ryoga Iwanami\textsuperscript{2} \and
Kaname Matsue\textsuperscript{3,4} \and
Hiroyuki Ochiai\textsuperscript{3}
}
\date{%
\textsuperscript{1}Global Center for Science and Engineering, Waseda University, 3-4-1 Okubo, Shinjuku-ku, Tokyo 169-8555, Japan\\
\textsuperscript{2}Graduate School of Fundamental Science and Engineering, Waseda University, 3-4-1 Okubo, Shinjuku-ku, Tokyo 169-8555, Japan\\
\textsuperscript{3}Institute of Mathematics for Industry, Kyushu University, 744 Motooka, Nishi-ku, Fukuoka 819-0395, Japan\\
\textsuperscript{4}International Institute for Carbon-Neutral Energy Research (WPI-I$^2$CNER), Kyushu University, 744 Motooka, Nishi-ku, Fukuoka 819-0395, Japan\\[1em]
\textsuperscript{*}Corresponding author. E-mail: tanaka@ims.sci.waseda.ac.jp
}

\begin{document}
\maketitle

\begin{abstract}
This paper presents a novel framework for enclosing solutions of Poisson's equation based on generalized sub- and super-solutions constructed using fundamental solutions. The conventional definition of sub- and super-solutions based on variational inequalities often fails for natural function classes such as piecewise linear functions and encounters theoretical difficulties in non-convex polygonal domains, where $H^2$ regularity is lost because of corner singularities. To overcome these limitations, we introduce the concept of ``Green-representable solutions'' utilizing test functions constructed from fundamental solutions. This framework enables a new formulation of sub- and super-solutions that permits rigorous pointwise evaluation. For one-dimensional problems, we derive explicit constructions of the test functions. For two-dimensional polygonal domains, we employ the Method of Fundamental Solutions to generate test functions. The approach is validated through numerical experiments in both settings, including non-convex polygons. The results demonstrate that the proposed method yields strict and accurate pointwise enclosures of the true solution, even for problems with discontinuous source terms or geometric singularities.
\end{abstract}
\vspace{1em}
\noindent
\textbf{Keywords:} Poisson's equation, Sub- and Super-solutions, Green's function, Solution enclosure, Pointwise evaluation, Corner singularities

\section{Introduction}\label{sec:intro}
The Poisson equation
\begin{equation}
    \label{eq:poisson-strong}
    \begin{cases}
    -\Delta u = f & \text{in }\Omega,\\
     u=0 & \text{on }\partial\Omega
    \end{cases}
\end{equation}
has been one of the most fundamental and universal partial differential equations since its formulation by Poisson in the early 1800s, appearing across an extraordinarily broad range of fields, from the natural sciences to engineering and social infrastructure.
These applications include Maxwell's equations in electrostatics \cite{jackson1999classical}, pressure Poisson equations in fluid mechanics \cite{landau1987fluid}, stress function methods in structural mechanics \cite{timoshenko1970theory}, steady-state problems in heat conduction \cite{carslaw1959conduction}, time-independent Schrödinger equations in quantum mechanics \cite{landau1977quantum}, determination of galactic gravitational potential fields in astrophysics \cite{binney2008galactic}, gradient domain processing in image processing \cite{perez2003poisson}, and potential field analysis of steady aquifer flow in groundwater analysis \cite{bear1972dynamics}.

From a theoretical perspective, as demonstrated in the classical textbook by Gilbarg \& Trudinger \cite{gilbarg2001elliptic}, this equation serves as a cornerstone of elliptic partial differential equation theory and has driven the development of modern analysis through foundational results such as the maximum principle, Harnack's inequality \cite{harnack1887grundlagen}, and Schauder estimates \cite{schauder1934uber}. Combined with Sobolev space theory systematized in Brezis's variational theory \cite{brezis2011functional}, it provides essential guidance for the analysis of more complex nonlinear problems. In computational science, the Poisson equation also plays a central role in the development of discretization methods and error analysis, serving as the theoretical foundation of finite element methods originating from Courant's variational approach \cite{courant1943variational}, as presented in Ciarlet's standard textbook \cite{ciarlet2002finite} and the modern theory of Brenner \& Scott \cite{brenner2008mathematical}. These developments form the basis for ensuring reliability and accuracy guarantees in engineering computations.

However, under complex geometric conditions such as polygonal domains with non-convex corners, including L-shaped domains and interfaces between different materials, the $H^2$ regularity of solutions is lost, and singularities arise, as established in Grisvard's theory \cite{grisvard2011elliptic} and Kondratiev's corner singularity analysis \cite{kondratiev1967boundary}. Because solution behavior depends sensitively on corner singular exponents, as clarified by Dauge's detailed analysis \cite{dauge1988elliptic}, standard finite element error analysis becomes difficult to apply.
This lack of regularity is closely related to practically significant phenomena, including stress concentration at aircraft structural corners, electric field enhancement at semiconductor device junctions, and discontinuities of physical quantities at composite material interfaces. The precise analysis of these effects remains challenging from both theoretical and computational perspectives \cite{seo2024sobolev}.

Regarding accuracy guarantees and error evaluation for solutions of Poisson equations, numerous studies have been actively pursued to date. In particular, efforts to rigorously verify the existence of solutions and to quantitatively evaluate their errors occupy an important position within the field of verified numerical computation. As a pioneering contribution in this area, the verification framework combining finite-dimensional projection with infinite-dimensional remainder evaluation proposed by Nakao et al. \cite{nakao1988numerical, nakao1990numerical} can be cited. This approach decomposes the solution into an approximate component obtained by finite element approximation and a rigorously evaluable remainder term, thereby guaranteeing the existence and uniqueness range of solutions and establishing a foundation for constructive error evaluation in elliptic boundary value problems.
In parallel, Plum addressed solution verification from a different perspective during the same period, developing methods to bound the inverse of linearized operators through rigorous eigenvalue evaluation of associated differential operators \cite{plum1991bounds}.
Yamamoto and Nakao further extended the framework of Nakao et al. to non-convex polygonal domains with reentrant corners, developing verification methods that explicitly account for corner singularities \cite{yamamoto1993numerical}. Additionally, Liu and Oishi proposed rigorous error evaluation methods for Poisson equations that underpin eigenvalue evaluation of Laplacians on arbitrarily shaped polygonal domains by constructing error constants $M_h$ based on interpolation error constants $C_0, C_1$, within a finite element framework \cite{liu2013verified}. This constructive error evaluation strategy has also been applied to the verification of semi-linear elliptic problems on arbitrary polygonal domains by Takayasu et al. \cite{takayasu2013verified}. Furthermore, Kobayashi proposed a method achieving more precise \textit{a priori} error evaluation by augmenting finite element bases with singular functions that explicitly represent solution singularities in non-convex domains \cite{kobayashi2009constructive}.

On the other hand, in the field of guaranteed a posteriori error estimation, substantial theoretical progress has been achieved through hypercircle methods based on the Prager-Synge theorem \cite{prager1947hypercircle} and through equilibrated flux reconstruction techniques. In particular, Vohralík et al. developed equilibrated flux methods that provide fully computable error upper bounds without mesh-dependent constants by employing H(div)-conforming flux reconstruction \cite{vohralik2008residual}, and subsequently opened pathways to practical adaptive mesh refinement through robust error estimation independent of polynomial degree \cite{ern2015polynomial}. Verfürth's comprehensive monograph \cite{verfurth1996review} offers a systematic classification of residual-based, recovery-based, and equilibrium residual methods, while Carstensen established a unified framework for error estimation across different finite element methods \cite{carstensen2005unifying}. More recently, Nakano and Liu proposed guaranteed local error evaluation methods applicable to subdomains even when solutions lack $H^2$ regularity, by applying hypercircle methods \cite{nakano2023guaranteed}. This approach is noteworthy for enabling direct and rigorous evaluation of solution accuracy in specific regions of interest. The growing interest in such pointwise and local evaluations underscores the relevance of the approach proposed in this paper, which is based on solution representation using fundamental solutions and rigorous enclosure at each point.
The studies reviewed above have primarily contributed to rigorous error evaluation in $L^2$ norms or energy norms, including $H^1$ norms, achieving important advances, particularly in solution existence guarantees and quantitative error evaluation within finite element frameworks.

Meanwhile, research has also focused on achieving precise evaluation in stronger norms, capturing solution behavior at specific points, and assessing solution accuracy within subdomains. Plum developed pointwise bounding methods to estimate solution values at each point based on explicit $H^2$ evaluation for second-order elliptic boundary value problems \cite{plum1992explicit}. Additionally, Nakao and Yamamoto proposed verification methods for nonlinear elliptic problems based on residual evaluation in $L^\infty$ norms. However, these methods rely on the assumption of $H^2$ regularity of solutions in their theoretical construction \cite{nakao1998numerical_linf}.

As a tool for enclosing solutions of Poisson equations from above and below, we focus on sub- and super-solutions. A conventional definition of super-solutions $\overline{u}$ is given by
\begin{equation}
    \label{eq:classical_super}
     (\nabla \overline{u},\nabla v) \ge (f,v) \quad \text{for all } v\in H^1_0(\Omega)\cap L^2_+(\Omega).
\end{equation}
Sub-solutions $\underline{u}$ are defined analogously with the inequality reversed \cite{evans2020partial}.
However, this variational-based definition of sub- and super-solutions may fail, from a theoretical standpoint, even for natural function classes such as piecewise linear functions. To illustrate this limitation explicitly, we consider the following problem.

\begin{problem}
Let $\Omega=(0,1)\subset\mathbb{R}$ and let $\overline{u}\in H^1(\Omega)$ be piecewise linear on $\Omega$. For $f\in L^2(\Omega)$, is there a general method to verify the following inequality?
\[ (\overline{u}',\phi') \ge (f,\phi) \quad \text{for all } \phi\in H^1_0(\Omega)\cap L^2_+(\Omega). \]
\end{problem}

As a negative answer to this question, when $f=1$, consider
\begin{equation}\label{eq:counterexample_supersol}
\overline{u}(x)=\begin{cases}
2ax & x\in[0,1/2],\\
2a(1-x) & x\in[1/2,1].
\end{cases}
\end{equation}
Using $\phi(x)=-x(x-0.5)^2(x-1)$, we obtain $(\overline{u}',\phi')=0<(f,\phi)$, showing that the super-solution condition is not satisfied for any $a\ge1/4$.
Figure \ref{fig:counterexam} illustrates this situation.
Even with increased subdivision points, similar counterexamples can be constructed by ensuring that $\phi$ vanishes at the subdivision points. Thus, even simple one-dimensional piecewise linear functions fail to satisfy the conventional super-solution conditions.
\begin{figure}
    \centering
    \includegraphics[width=0.6\linewidth]{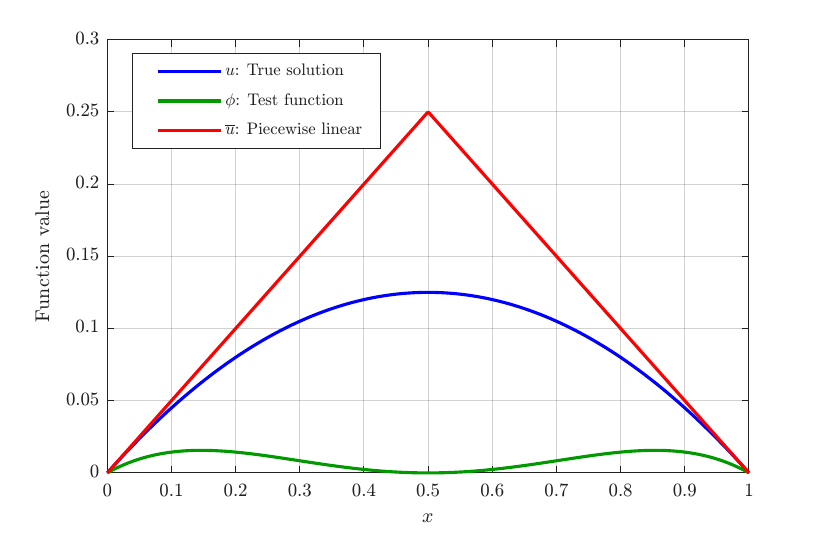}
    \caption{Linear functions do not satisfy the super-solution condition. Blue: true solution of \eqref{eq:poisson-strong} when $f=1$; Green: $\phi(x)=-x(x-0.5)^2(x-1)$; Red: piecewise linear function \eqref{eq:counterexample_supersol}}
    \label{fig:counterexam}
\end{figure}
This observation provides strong motivation for the present research, namely, "to generalize the concept of sub- and super-solutions and to reformulate them in a more tractable and verifiable form."

To overcome the limitations of the classical framework described above, particularly those demonstrated by the counterexamples, this paper introduces a new solution representation called ``Green-representable solutions,'' based on test functions constructed using fundamental solutions from a weak solution perspective. The main problem addressed in this study is to find $u\in H^1_0(\Omega)$ satisfying the weak formulation
\begin{equation}
    \label{eq:main}
    \int_{\Omega} \nabla u\cdot\nabla v\,dx = \int_{\Omega} f\,v\,dx \quad\text{for all } v\in H^1_0(\Omega).
\end{equation}
The regularity conditions imposed on $f$ in Section \ref{subsec:fundasol-integrability}. By developing the concept of Green-representable solutions, it becomes possible to formulate new types of sub- and super-solutions that are more direct and verifiable, thereby replacing conventional variational inequalities. This objective constitutes the central contribution of the present research. Through this new framework, the difficulties inherent in constructing conventional sub- and super-solutions via variational inequalities can be systematically addressed.

In constructing the proposed new framework for sub- and super-solutions, another important property of Green-representable solutions is that rigorous pointwise evaluation of solutions naturally emerges. This pointwise evaluation is an important byproduct identified during the pursuit of the main objective of generalizing sub- and super-solutions, and it is itself highly useful for precisely capturing local properties of solutions.
As related research on local error evaluation, Nakano and Liu \cite{nakano2023guaranteed} proposed guaranteed local error evaluation for finite element solutions based on hypercircle theory, demonstrating applicability to problems that do not assume $H^2$ regularity. While their method provides evaluation in $L^2$ norms or energy norms, the present research differs in its ability to achieve rigorous evaluation of $L^\infty$ norms and pointwise values through direct solution representation based on fundamental solutions.

The main contributions of this paper are as follows:
\begin{enumerate}
    \item Introduction of the concept of "Green-representable solutions" based on fundamental solutions and formulation of new sub- and super-solutions using them
    \item Construction and verification of sub- and super-solutions for broader function classes, including piecewise linear functions that cannot be handled by conventional sub- and super-solution definitions
    \item Presentation of rigorous pointwise evaluation methods that arise naturally from the proposed framework
    \item Extension of the theory to two-dimensional polygonal domains that may contain corner points, together with solution evaluation in such domains
    \item Practical construction of test functions using the Method of Fundamental Solutions (MFS)
\end{enumerate}
Furthermore, by employing the MFS to construct test functions, this work offers a novel perspective on MFS performance. The accuracy of the resulting solution enclosures provides an indirect yet quantitative measure of MFS approximation quality, highlighting a new application of MFS in verified numerical computation and motivating further research on its optimization for high-precision error estimation.

This paper is organized as follows. First, in Section \ref{sec:rep-green}, we establish a solution representation theory based on fundamental solutions. We introduce the concepts of local and global Green-representability and construct a theoretical framework for Lipschitz domains of arbitrary dimension.
 While this theory holds for general dimensions, special results regarding solution regularity and boundary condition treatment are obtained in one- and two-dimensional settings.
 In Section \ref{sec:onedim}, we develop a dedicated theoretical framework for one-dimensional problems, and in Section \ref{sec:onedim}, we verify the effectiveness of the proposed method through numerical experiments in one dimension. Section \ref{sec:twodim} extends the framework to two-dimensional problems, particularly detailing solution evaluation in non-convex polygonal domains. Numerical experiments are presented for constant, polynomial, and non-polynomial source terms in square and L-shaped domains. Finally, Section \ref{sec:conclusion} provides comparisons with related research and discusses future challenges.

\section{Solution Representation via fundamental Solutions}\label{sec:rep-green}
In this section, we establish a framework for representing solutions of the boundary value problem \eqref{eq:main} by means of fundamental solutions. 
The theory developed here is general and applies to Lipschitz domains in arbitrary dimensions. However, as shown later, the one- and two-dimensional cases yield particularly elegant results owing to their special properties related to solution regularity.
We begin in Subsection~\ref{subsec:fundasol-integrability} by introducing the necessary notation and integrability conditions for fundamental solutions.
Subsection~\ref{subsec:testfunctions} describes the construction of test functions that serve as representation kernels.
In Subsections~\ref{subsec:local-Green} and~\ref{subsec:global-Green}, we introduce two key concepts---local and global Green-representability---which together provide a unified framework for analyzing solutions across arbitrary dimensions, with particular emphasis on their implications in one and two dimensions.

\subsection{Fundamental solutions and integrability conditions}\label{subsec:fundasol-integrability}
We now introduce the framework in terms of fundamental solutions. Let $\Omega$ be a bounded Lipschitz domain in $\mathbb{R}^N$. For the construction of test functions, we consider the following components:
\begin{itemize}
\item A fundamental solution $\fundsol(s,x)$ satisfying $-\Delta\fundsol(s,x)=\delta(x-s)$, where $\delta$ denotes the Dirac delta function. The explicit form depends on the spatial dimension
\begin{equation}
\fundsol(s,x) = \begin{cases}
-\frac{|x-s|}{2} & \text{for } N=1 \\
-\frac{1}{2\pi}\log|x-s| & \text{for } N=2 \\
\frac{1}{(N-2)\omega_N|x-s|^{N-2}} & \text{for } N\geq 3
\end{cases}
\end{equation}
where $\omega_N$ denotes the surface area of the unit sphere in $\mathbb{R}^N$.
\item An interior point $\intpt \in \Omega$ at which the solution is to be evaluated.
\item A finite set of exterior points $s_i \in \mathbb{R}^N \setminus \overline{\Omega}$ $(i=1,\ldots,n)$, which are used to construct an approximation of the fundamental solution.
\end{itemize}

For the well-posedness of the proposed framework, it is necessary to ensure that certain pairings between $f$ and fundamental solutions are finite. We therefore identify the minimal regularity requirements imposed on $f$, depending on the spatial dimension.

\begin{assumption}[Integrability of $f$]\label{assum:f-integrability}
Throughout this paper, we assume that $f$ satisfies:
\begin{itemize}
\item For $N=1$: $f \in L^1(\Omega)$
\item For $N \geq 2$: $f \in L^p(\Omega)$ for some $p > N/2$
\end{itemize}
\end{assumption}

For a function $f$ these integrability conditions and $\fundsol_{\intpt}(x):=\fundsol(\intpt,x)$, we define the coupling
\begin{equation}
\langle f,\fundsol_{\intpt} \rangle := \int_\Omega f(x)\fundsol_{\intpt}(x)dx,
\label{eq:coupling}
\end{equation}
whenever the integral exists.

\begin{lemma}[Integrability of the pairing]\label{lem:integrability}
Under Assumption \ref{assum:f-integrability}, the pairing $\langle f,\fundsol_{\intpt} \rangle$ defined in \eqref{eq:coupling} is well-defined and finite for any $\intpt \in \Omega$.
\end{lemma}

\begin{proof}
For $N=1$, this property is readily verified since the fundamental solution belongs to $L^\infty(\Omega)$.
For $N \geq 2$, the condition follows from Hölder's inequality
\begin{equation}
|\langle f,\fundsol_{\intpt} \rangle | \leq \|f\|_{L^p} \|\fundsol_{\intpt}\|_{L^{p'}}
\end{equation}
where $\frac{1}{p} + \frac{1}{p'} = 1$.
Examining the boundedness of the latter norm, we obtain
$N-1-(N-2)p' > -1$, which yields $p > \frac{N}{2}$.
\end{proof}

When $f$ satisfies this regularity condition, the solution of \eqref{eq:main} is locally continuous in the interior of $\Omega$.
Consequently, $u(s)$ is well-defined for any $s \in \Omega$.
In the one-dimensional case, this follows from the Sobolev embedding $H^1_0(\Omega) \hookrightarrow C^0(\Omega)$; see \cite[Theorem 4.12]{adams2003sobolev} for details.
For dimensions $N \geq 2$, when $p > N/2$, we have $u \in W^{2,p}_{\mathrm{loc}}(\Omega) \hookrightarrow C^{0,\alpha}_{\mathrm{loc}}(\Omega)$ with $\alpha = 2-\frac{N}{p}>0$, which ensures the desired continuity. A comprehensive treatment of these embeddings can be found in \cite[Theorem 7.26]{gilbarg2001elliptic}.
These regularity conditions are stronger than the requirement $f \in H^{-1}(\Omega)$ arising from the Lax-Milgram theorem for the weak formulation of \eqref{eq:main}, where $H^{-1}(\Omega)$ denotes the dual space of $H^1_0(\Omega)$. However, for dimensions $N \leq 3$, they are satisfied by the commonly assumed condition $f \in L^2(\Omega)$, indicating that the imposed requirements are not overly restrictive for most practical applications.

\subsection{Construction of test functions}\label{subsec:testfunctions}
For the construction of test functions, we introduce the following general form.
\begin{construction}[Test function construction]\label{con:test-function}
For an evaluation point $\intpt \in \Omega$ and a non-zero real number $\intcoef$ (which may depend on $\intpt$), we define a test function by:
\begin{equation}
\phi_{\intpt}(x):=\intcoef\fundsol(\intpt,x)+\harmonic(x)\label{eq:test-function}
\end{equation}
where $\harmonic$ depends on $\intpt$ and is harmonic and of class $C^2$ in a domain $U$ containing $\overline{\Omega}$ and a neighborhood. That is, there exists an open set $U \supset \overline{\Omega}$ such that $\harmonic \in C^2(U)$ and $\Delta \harmonic = 0$ in $U$.
\end{construction}
We denote the set of all such test functions by $\testFuncSet(\Omega)$. We omit $(\Omega)$ and simply write $\testFuncSet$ when there is no risk of confusion.
For practical implementation, we specify $\harmonic(x)$ as
\begin{equation}
\harmonic(x)=\sum_{i=1}^{n}a_{i}\fundsol(s_i,x)+C
\label{eq:harmonic-part}
\end{equation}
where source points $\{s_i\}_{i=1}^n \subset \mathbb{R}^N \setminus \overline{\Omega}$, coefficients $a_i \in \mathbb{R}$ $(i=1,\ldots,n)$, and constant $C \in \mathbb{R}$ are chosen to approximate the homogeneous Dirichlet boundary condition in \eqref{eq:main} as closely as possible.
Note that $\harmonic$ and $\phi_{\intpt}$ depend not only on $\intpt$ but also on the parameters $a_i$, $C$, and $s_i$. In this paper, to avoid notational complexity, we retain the notations $\harmonic$ and $\phi_{\intpt}$ while explicitly indicating only the primary parameter $\intpt$.

Having established the construction of test functions, we now turn to a key coupling that plays a central role in the subsequent analysis.
For $u \in H^1(\Omega)$ and any $\phi_{\intpt} \in \testFuncSet$, we define the coupling by
\begin{equation}
\label{eq:coupling_def}
\left\langle\nabla u,\nabla\phi_{\intpt}\right\rangle:=-\int_{\Omega}u\Delta\phi_{\intpt}dx+\int_{\partial\Omega}u\frac{\partial\phi_{\intpt}}{\partial n}d\belem=\intcoef u(\intpt)+\int_{\partial\Omega}u\frac{\partial\phi_{s}}{\partial n}d\belem,
\end{equation}
where $\Delta\phi_{\intpt}$ is interpreted in the distributional sense as $-\Delta \phi_{\intpt} = \intcoef\delta(x-\intpt)$.
The following lemma is fundamental to the proposed framework, as it allows the above coupling to be interpreted as a standard Lebesgue integral.

\begin{lemma}[Interpretation of the coupling as a Lebesgue integral]\label{lem:coupling_interpretation}
For $N \geq 2$ and $f \in L^p(\Omega)$ with $p > N/2$, let $u \in H^1_0(\Omega)$ be a solution of \eqref{eq:main}.
Suppose that $\phi_{\intpt} \in \testFuncSet$ is a test function with a fixed interior point $\intpt \in \Omega$.
Then, the coupling $\langle \nabla u, \nabla \phi_{\intpt} \rangle$ admits the Lebesgue integral representation
\begin{equation}
\label{eq:nabla_u_phi_coupling}
\langle \nabla u, \nabla \phi_{\intpt} \rangle = \int_{\Omega} \nabla u \cdot \nabla \phi_{\intpt} \, dx.
\end{equation}
This representation is equivalent to \eqref{eq:coupling_def} via Green's first identity.
\end{lemma}

\begin{proof}
Since $f \in L^p_{\mathrm{loc}}(\Omega)$ with $p > N/2$, elliptic regularity theory (see, e.g., \cite{evans2020partial}), we have $u \in W^{2,p}_{\mathrm{loc}}(\Omega)$.
By the Sobolev embedding theorem, this yields $u \in W^{1,q}_{\mathrm{loc}}(\Omega)$ and hence $\nabla u \in L^q_{\mathrm{loc}}(\Omega)$ for $\frac{1}{q} = \frac{1}{p} - \frac{1}{N}$. Since $p > N/2$, we have $q = \frac{pN}{N-p} > N$.

We next examine the integrability of $\nabla \phi_{\intpt}$. Because $\phi_{\intpt}$ contains the fundamental solution $\fundsol(\intpt,x)$, its gradient $\nabla \phi_{\intpt}$ exhibits a singularity of order $|x-\intpt|^{1-N}$ near $\intpt$. For the Lebesgue integral $\int_{\Omega} \nabla u \cdot \nabla \phi_{\intpt} dx$ to be well-defined via Hölder's inequality, it is sufficient that $\nabla \phi_{\intpt} \in L^{q'}(\Omega)$, where $q'$ is the conjugate exponent of $q$ (i.e., $\frac{1}{q} + \frac{1}{q'} = 1$). 
Using spherical coordinates centered at $\intpt$, this requirement reduces to the condition 
\begin{equation}
\int_{B_r(\intpt)} |\nabla \phi_{\intpt}(x)|^{q'} dx < \infty
\end{equation}
for a sufficiently small ball $B_r(\intpt)$ around $\intpt$. Since $|\nabla \phi_{\intpt}(x)| \sim |x-\intpt|^{1-N}$ near $\intpt$, the above integral becomes
\begin{equation}
\int_0^r \rho^{(1-N)q'} \rho^{N-1} d\rho = \int_0^r \rho^{(1-N)q' + N-1} d\rho < \infty.
\end{equation}
This improper integral converges if and only if $(1-N)q' + N-1 > -1$, which is equivalent to $q' < \frac{N}{N-1}$, or equivalently $q > N$.
Since this condition has already been established, the integral $\int_{\Omega} \nabla u \cdot \nabla \phi_{\intpt} dx$ is well-defined and finite. The equivalence with \eqref{eq:coupling_def} then follows from Green's first identity.  
\end{proof}

The above argument applies only when the point $\intpt$ is fixed in the interior of the domain. To obtain estimates that are uniform with respect to $\intpt$, including points approaching the boundary, a stronger regularity assumption $u \in W^{1,q}(\Omega)$ with $q > N$ is required.
Under this assumption, $\nabla u \in L^q(\Omega)$, while the gradient of the fundamental solution $\nabla \phi_{\intpt}$ belongs to $L^{q'}_{loc}(\mathbb{R}^N)$, where $1/q + 1/q' = 1$. Because $q > N$ implies $q' < N/(N-1)$, this condition precisely matches the integrability threshold for the $|x|^{1-N}$ singularity.
Consequently, the coupling is well-defined uniformly in $\intpt$ via H\"older's inequality.

Finally, by Morrey's inequality, we have the embedding
\begin{equation}
W^{1,q}(\Omega) \hookrightarrow C^{0,\alpha}(\overline{\Omega}), \quad \alpha = 1 - \frac{N}{q} > 0,
\end{equation}
which ensures that $u(\xi)$ is welldefined for every $\xi \in \overline{\Omega}$. Further details on these embeddings can be found in \cite{adams2003sobolev,gilbarg2001elliptic}.

When the domain boundary is sufficiently smooth (e.g., $C^{1,1}$) and Assumption \ref{assum:f-integrability} holds, elliptic regularity theory yields $u \in W^{2,p}(\Omega)$, which by the Sobolev embedding theorem implies $u \in W^{1,q}(\Omega)$ for $q > N$. 
For non-smooth domains, although there has been considerable research on the conditions under which such regularity may still hold, a complete characterization remains an open problem \cite{grisvard2011elliptic,seo2024sobolev}.
% However, as we will see in Section \ref{sec:twodim}, for two-dimensional polygonal domains, even when the domain is non-convex (i.e., when $u$ does not possess $W^{2,p}$ regularity), we still have $u \in W^{1,q}(\Omega)$ for $q > 2$.
On the other hand, as detailed in Section~\ref{sec:twodim}, in the case of two-dimensional polygonal domains, it is known that even if the domain is non-convex and $u$ does not possess $W^{2,p}$ regularity, one can still obtain $u \in W^{1,q}(\Omega)$ for a certain range of exponents $q > 2$.
% by combining Grisvard's theory with the Sobolev embedding theorem. In particular, when $f \in L^2(\Omega)$, this regularity holds for any $q < 2/(1-\pi/\omega_{\max})$, where $\omega_{\max} > \pi$ denotes the largest interior angle of the polygon. This estimate shows that although the obtained regularity is better than $W^{1,2}$, it is essentially limited by the geometric properties of the domain.

To give a rigorous meaning to boundary fluxes, we recall the notion of the weak Neumann trace.
\begin{definition}[Weak Neumann trace]\label{def:weak_neumann_trace}
Let $u\in H^1(\Omega)$ satisfy $-\Delta u = f$ in $\Omega$ in the sense of distributions, that is,
\[
 \int_{\Omega} \nabla u \cdot \nabla \varphi \, dx = \int_{\Omega} f\, \varphi \, dx \quad\text{for all } \varphi \in C_c^{\infty}(\Omega).
\]
Here, $C_c^{\infty}(\Omega)$ denotes the space of smooth functions with compact support in $\Omega$, i.e., infinitely differentiable functions that vanish outside a bounded subset of $\Omega$.
For $\psi \in H^{\frac{1}{2}}(\partial\Omega)$, the weak Neumann trace of $u$, denoted by $\frac{\partial u}{\partial n} \in H^{-\frac{1}{2}}(\partial\Omega)$, is defined by
\begin{equation}
 \left\langle \frac{\partial u}{\partial n},\, \psi \right\rangle_{H^{-\frac{1}{2}},H^{\frac{1}{2}}}
 := \int_{\Omega} \nabla u \cdot \nabla v \, dx - \int_{\Omega} f\, v \, dx
 \label{eq:weak_neumann_trace}
\end{equation}
for any $v\in H^{1}(\Omega)$ such that $\operatorname{tr} v = \psi$, where $\operatorname{tr}:H^{1}(\Omega)\to H^{\frac{1}{2}}(\partial\Omega)$ denotes the trace operator.
\end{definition}

When $u \in H^2(\Omega)$, the weak Neumann trace coincides with the classical normal derivative, and the definition \eqref{eq:weak_neumann_trace} reduces to Green's theorem
\[
 \int_{\partial\Omega} \frac{\partial u}{\partial n} \psi \, d\belem = \int_{\Omega} \nabla u \cdot \nabla v \, dx - \int_{\Omega} f\, v \, dx = \int_{\Omega} (-\Delta u - f) v \, dx + \int_{\partial\Omega} \frac{\partial u}{\partial n} v \, d\belem.
\]

\subsection{Local Green-representability}\label{subsec:local-Green}
Having completed the necessary preparations, we now introduce the concept of local Green-representability, which constitutes a cornerstone of the proposed theoretical framework. This concept is formulated in terms of fundamental solutions and provides a basis for representing solutions at fixed interior points.

We begin by defining local Green-representability for each fixed interior point $\intpt$. This localized notion serves as the foundation for the subsequent development of a global representation theory.

\begin{definition}[Local Green-representability]\label{def:green_representable_local}
For a fixed $\intpt \in \Omega$, let $\phi_{\intpt} \in \mathcal{T}_{\intpt}$ be a test function of the form \eqref{eq:test-function}.
A solution $u \in H^1_0(\Omega)$ of \eqref{eq:main} is said to be Green-representable with respect to $\phi_{\intpt}$ (or simply $\phi_{\intpt}$-Green-representable for short) if
\begin{equation}
\label{eq:green_rep}
\intcoef u(\intpt) = \langle f,\phi_{\intpt} \rangle + \Big\langle \frac{\partial u}{\partial n},\, \gamma(\phi_{\intpt}) \Big\rangle_{H^{-\frac{1}{2}}(\partial\Omega),H^{\frac{1}{2}}(\partial\Omega)}.
\end{equation}
\end{definition}

The freedom in choosing $\phi_{\intpt}$ in Definition \ref{def:green_representable_local} or $\testFuncMap$ in Definition \ref{def:green_representable_global} plays a pivotal role in the evaluation of the solution $u$, particularly in higher dimensions where $N\geq 2$. The details will be discussed in Subsections \ref{subsec:practical-remarks} and \ref{subsec:test_function_construction}.
As stated in the following theorem, under the regularity condition $f \in L^p(\Omega)$ with $p > N/2$ and the assumption that $\Omega$ is a bounded Lipschitz domain, any test function $H_{\intpt}$ of the form \eqref{eq:test-function} from Construction \ref{con:test-function} is admissible.

\begin{theorem}[Local Green-representability for regular data]\label{thm:local_green_rep}
Let $\Omega \subset \mathbb{R}^N$ $(N\ge 2)$ be a bounded Lipschitz domain and let $f\in L^p(\Omega)$ with $p>\frac{N}{2}$.
For every fixed $\intpt\in\Omega$ and test function $\phi_{\intpt} \in \mathcal{T}_{\intpt}$ of the form \eqref{eq:test-function}, a solution $u\in H^1_0(\Omega)$ of \eqref{eq:main} is $\phi_{\intpt}$-Green-representable.
\end{theorem}

\begin{proof}
Fix $\intpt\in\Omega$ and set $\psi:=\gamma(\phi_{\intpt})$. 
We have $\nabla\phi_{\intpt} \in L^2(\Omega \setminus B_r(\intpt))$ for all $r > 0$. Since $p > \frac{N}{2}$, it follows that $f\,\phi_{\intpt} \in L^1(\Omega)$.
For sufficiently small $\varepsilon>0$ with $B_{2\varepsilon}(\intpt)\subset\Omega$, choose a smooth cutoff function $0\le\chi_{\varepsilon}\le1$ satisfying
\[
\chi_{\varepsilon}=0\text{ on }B_{\varepsilon}(\intpt),\qquad
\chi_{\varepsilon}=1\text{ on }\Omega\setminus B_{2\varepsilon}(\intpt),\qquad
|\nabla\chi_{\varepsilon}|\le C/\varepsilon,
\]
and define $v_{\varepsilon}:=\chi_{\varepsilon}\phi_{\intpt}\in H^1(\Omega)$. Because $\chi_{\varepsilon}\equiv1$ in a neighborhood of $\partial\Omega$, we have $\gamma(v_{\varepsilon})=\psi$. By the definition of the normal derivative,
\[
\Big\langle \frac{\partial u}{\partial n},\psi\Big\rangle_{H^{-\frac{1}{2}},H^{\frac{1}{2}}}
 = \int_{\Omega}\nabla u\cdot\nabla v_{\varepsilon}\,dx-\int_{\Omega}f\,v_{\varepsilon}\,dx.
\]
Expanding the right-hand side yields
\[
\int_{\Omega}\chi_{\varepsilon}\nabla u\cdot\nabla\phi_{\intpt}\,dx
 +\int_{\Omega}\phi_{\intpt}\nabla u\cdot\nabla\chi_{\varepsilon}\,dx
 -\int_{\Omega}f\,\chi_{\varepsilon}\phi_{\intpt}\,dx.
\]
The first and third terms converge, respectively to $\int_{\Omega}\nabla u\cdot\nabla\phi_{\intpt}\,dx$ and $\int_{\Omega}f\,\phi_{\intpt}\,dx$ by dominated convergence. For the middle term, observe that $\operatorname{supp}\nabla\chi_{\varepsilon}\subset A_{\varepsilon}:=B_{2\varepsilon}(\intpt)\setminus B_{\varepsilon}(\intpt)$ Hence,
\[
\Big|\int_{\Omega}\phi_{\intpt}\nabla u\cdot\nabla\chi_{\varepsilon}\,dx\Big|
\le \|\nabla u\|_{L^2(A_{\varepsilon})}\,\|\phi_{\intpt}\nabla\chi_{\varepsilon}\|_{L^2(A_{\varepsilon})}
\le \|\nabla u\|_{L^2(A_{\varepsilon})}\,\frac{C}{\varepsilon}\,\|\phi_{\intpt}\|_{L^2(A_{\varepsilon})}.
\]
Since $u \in H^1(\Omega)$, by interior regularity $\nabla u$ is continuous near $\intpt$, so $\|\nabla u\|_{L^2(A_{\varepsilon})} = O(\varepsilon^{N/2})$.
Since $\phi_{\intpt}$ behaves like the fundamental solution near $\intpt$, we have $\|\phi_{\intpt}\|_{L^2(A_{\varepsilon})}=O(\varepsilon^{(4-N)/2})$ for $N\ge3$ and $O(\varepsilon|\log\varepsilon|)$ for $N=2$.
Consequently, the factor $(C/\varepsilon)\|\phi_{\intpt}\|_{L^2(A_{\varepsilon})} = O(\varepsilon^{(2-N)/2})$ for $N \geq 3$ and $O(|\log\varepsilon|)$ for $N=2$, which diverges as $\varepsilon \to 0$.
Nevertheless, multiplying both factors yields $O(\varepsilon)$ for $N \geq 3$ and $O(\varepsilon|\log\varepsilon|)$ for $N=2$, both of which vanish as $\varepsilon \to 0$.
Hence, the mixed term converges to $0$. Taking the limit yields
\[
\Big\langle \frac{\partial u}{\partial n},\psi\Big\rangle_{H^{-\frac{1}{2}},H^{\frac{1}{2}}}
 =\int_{\Omega}\nabla u\cdot\nabla\phi_{\intpt}\,dx-\int_{\Omega}f\,\phi_{\intpt}\,dx.
\]
Combining this identity with Lemma~\ref{lem:coupling_interpretation} yields \eqref{eq:green_rep}, completing the proof.
\end{proof}

\begin{remark}
When $\Omega$ is a bounded Lipschitz domain, a Green's function exists for each source point $\intpt \in \Omega$ \cite{dahlberg1979lq}. This preceding theorem may be viewed as a generalization of this classical result, or at least as being consistent with it. Indeed, if one chooses $\phi_{\intpt}$ to be the Green's function, then $\gamma(\phi_{\intpt}) = 0$ on $\partial\Omega$, and the boundary integral term vanishes, thereby recovering the classical representation formula. However, explicit forms of Green's functions are rarely available in practice.
\end{remark}

In polytopic domains, the boundary may contain singular points (edges, vertices) where the normal vector is not well-defined.
We now consider an $N$-dimensional bounded polytopic domain $\Omega$. That is, $\Omega$ is a bounded domain composed of finitely many (hyper) planar pieces. For this domain, solutions of \eqref{eq:main} may exhibit singular behavior(i.e., loss of $W^{2,p}$ regularity) at a finite number of ``corners'' on the boundary.
We denote these singular points (corners) by $x_c$ in what follows.
Nevertheless, by interior regularity theory, we have $u \in W^{2,p}_{\mathrm{loc}}(\Omega)$, where $p > N/2$ is determined by the regularity of $f$.

\begin{theorem}\label{thm:thm:Nd_representation_local}
Let $\Omega \subset \mathbb{R}^N$ ($N \geq 2$) be a bounded $N$-dimensional polytopic domain and let $f \in L^p(\Omega)$ with $p > N/2$. Let $u \in H^1_0(\Omega)$ be the weak solution of \eqref{eq:main}.
Then, for any fixed $\intpt\in \Omega$ and any $\phi_{\intpt} \in \mathcal{T}_{\intpt}$ of the form \eqref{eq:test-function}, the Cauchy principal value $\mathrm{p.v.}\int_{\partial\Omega} \frac{\partial u}{\partial n}\gamma(\phi_{\intpt})\,d\belem$ exists and satisfies
\[
    \intcoef u(\intpt) = \langle f,\phi_{\intpt} \rangle + \mathrm{p.v.}\int_{\partial\Omega} \frac{\partial u}{\partial n}\gamma(\phi_{\intpt})\,d\belem.
\]
\end{theorem}

\begin{proof}
For simplicity, we assume that there is a single singular point $x_c$, although our theory extends naturally to the case of multiple singular points. Near the corner $x_c$ of $\partial\Omega$, let $\Omega_\varepsilon$ denote an $\varepsilon$-neighborhood of $x_c$.
Note that depending on the geometry of $\Omega$, the intersection $\Omega\cap\Omega_{\varepsilon}$ may be disconnected. In such cases, the following argument is applied to each connected component of $\Omega\cap\Omega_{\varepsilon}$. Recall that we are considering an arbitrary but fixed point $\intpt \in \Omega$. By interior regularity, the coupling $\langle \nabla u, \nabla \phi_{\intpt} \rangle$ is well-defined as a Lebesgue integral (see Subsection \ref{subsec:testfunctions}).

Furthermore, by Construction \ref{con:test-function}, for sufficiently small $\varepsilon$, both $\phi_{\intpt} \in \testFuncSet$ and its derivatives remain bounded in $\Omega_\varepsilon$. This boundedness property is used in the estimates below.

Since $\langle \nabla u, \nabla \phi_{\intpt} \rangle$ is well-defined as a Lebesgue integral, the Green representation formula applies.
\begin{equation}
\label{eq:green_rep_local}
\int_{\partial(\Omega\setminus\Omega_{\varepsilon})} \frac{\partial u}{\partial n}\phi_{\intpt} d\belem
= \langle \nabla u, \nabla \phi_{\intpt} \rangle - \int_{\Omega\setminus\Omega_{\varepsilon}} f\phi_{\intpt} dx - \int_{\Omega\cap\Omega_{\varepsilon}} \nabla u \cdot \nabla \phi_{\intpt} dx.
\end{equation}

We proceed to show that, as $\varepsilon \to 0$
\begin{enumerate}
\item $\int_{\Omega\setminus\Omega_{\varepsilon}}f\cdot\phi_{\intpt} dx \to \int_{\Omega}f\cdot\phi_{\intpt} dx$;
\item $\int_{\Omega\cap\Omega_{\varepsilon}}\nabla u\cdot\nabla\phi_{\intpt} dx \to 0$.
\end{enumerate}

{\bf First Limit:} For the first limit, by H\"older's inequality,
\[
\left|\int_{\Omega}f\cdot\phi_{\intpt}dx-\int_{\Omega\setminus\Omega_{\varepsilon}}f\cdot\phi_{\intpt}dx\right|
\leq \|\phi_{\intpt}\|_{L^\infty(\Omega_\varepsilon)} \|f\|_{L^p(\Omega)} |\Omega_\varepsilon|^{1-\frac{1}{p}}
\leq C_1\|f\|_{L^p(\Omega)} \varepsilon^{N(1-\frac{1}{p})}.
\]
The condition $p > N/2 \geq 1$ ensures that the exponent of $\varepsilon$ is positive, and therefore this term vanishes as $\varepsilon \to 0$.

{\bf Second Limit:} For the second limit, we apply Green's first identity. Since $\Delta \phi_{\intpt} = 0$ in $\Omega \cap \Omega_\varepsilon$, we have
\[
    \left|\int_{\Omega\cap\Omega_{\varepsilon}}\nabla u\cdot\nabla\phi_{\intpt} dx\right|
    = \left|\int_{\partial(\Omega\cap\Omega_{\varepsilon})}u\frac{\partial\phi_{\intpt}}{\partial n}d\belem\right|.
\]
Because $\partial\phi_{\intpt}/\partial n$ is bounded, setting $\left\|\frac{\partial\phi_{\intpt}}{\partial n}\right\|_{L^{\infty}}=C_2'$, we have
\[
    \left|\int_{\partial(\Omega\cap\Omega_{\varepsilon})}u\frac{\partial\phi_{\intpt}}{\partial n}d\belem\right|
    \leq C_2' \|u\|_{L^1(\partial(\Omega\cap\Omega_{\varepsilon}))}.
\]
We now estimate the boundary trace using Lemma~\ref{lem:trace_value_general_n_dim} with $p=1$. For $N \geq 3$, applying the lemma with $q=2$ (i.e., using $H^1$ regularity) yields the exponent $(N-1)/1 - N/2 = N/2 - 1 \geq 1/2 > 0$, so for some constant $C_2''>0$,
\[
    \|u\|_{L^1(\partial(\Omega\cap\Omega_{\varepsilon}))} \le C_2'' \varepsilon^{\frac{N}{2}-1} \|u\|_{H^1(\Omega)}.
\]
For $N=2$, the exponent $N/2-1=0$ does not guarantee convergence. However, as shown in Subsection~\ref{subsec:2d-regularity}, the solution on a two-dimensional polygonal domain satisfies $u \in W^{1,q}(\Omega)$ for some $q > 2$. Using this enhanced regularity, the exponent becomes $(N-1)/1 - N/q = 1 - 2/q > 0$, yielding
\[
    \|u\|_{L^1(\partial(\Omega\cap\Omega_{\varepsilon}))} \le C_2'' \varepsilon^{1-\frac{2}{q}} \|u\|_{W^{1,q}(\Omega)}.
\]
In both cases, the exponent is positive, so the last term converges to zero as $\varepsilon \to 0$.

Combining these convergence results with \eqref{eq:green_rep_local}, we conclude that:
\[
    \mathrm{p.v.}\int_{\partial\Omega} \frac{\partial u}{\partial n}\gamma(\phi_{\intpt}) d\belem
    = \langle \nabla u, \nabla \phi_{\intpt} \rangle - \int_{\Omega} f\phi_{\intpt} dx.
\]
\end{proof}

\subsection{Global Green-representability}\label{subsec:global-Green}
Having established local Green-representability (Theorem \ref{thm:local_green_rep}), we now introduce its global counterpart.  
The transition from the local to the global setting can be viewed as defining this property as a corollary of Theorem \ref{thm:local_green_rep}, but with careful attention to the regularity requirements imposed on the solution.
The principal difficulty arises from the fact that the evaluation point $\intpt$ may approach arbitrarily close to the boundary $\partial\Omega$.
Recall that the test function $\phi_{\intpt}$ behaves like the fundamental solution, and in particular its gradient exhibits a singularity of order $|x-\intpt|^{1-N}$ near $\intpt$.
When $\intpt$ is fixed in the interior, this singularity is isolated from the boundary. However, in the global setting where $\intpt$ may approach the boundary (and in particular, the singular points $x_c$ of the boundary), the interaction between the singularity of $\nabla\phi_{\intpt}$ and the potentially singular behavior of $\nabla u$ near $x_c$ becomes nontrivial.
Specifically, if $u$ only belongs to $H^1(\Omega)$ (i.e., $\nabla u \in L^2$), the product $|\nabla u \cdot \nabla \phi_{\intpt}|$ may fail to be Lebesgue integrable when $\intpt$ is sufficiently close to a singular boundary point $x_c$. Since the boundary term $\langle \frac{\partial u}{\partial n}, \gamma(\phi_{\intpt}) \rangle$ in \eqref{eq:green_rep} is defined through this coupling (see Definition \ref{def:weak_neumann_trace}), such a failure of integrability renders the boundary term undefined. Consequently, the representation formula \eqref{eq:green_rep} no longer holds in this setting.

To avoid this difficulty and ensure that the boundary term is well-defined uniformly with respect to $\intpt$, we require stronger regularity on $u$.
As shown in Subsection \ref{subsec:testfunctions}, if $u \in W^{1,q}(\Omega)$ for some $q > N$, then by H\"older's inequality the coupling $\langle \nabla u, \nabla \phi_{\intpt} \rangle$ is guaranteed to be well-defined uniformly in $\intpt$.
Motivated by this observation, we now define global Green-representability and state the corresponding corollary.

\begin{definition}[Global Green-representability]\label{def:green_representable_global}
Let $\testFuncMap: \intpt \mapsto \phi_{\intpt}$ be a mapping that assigns to each point $\intpt \in \Omega$ a test function $\phi_{\intpt} \in \mathcal{T}_{\intpt}$ of the form \eqref{eq:test-function}. 
A solution $u$ of \eqref{eq:main} is said to be globally Green-representable with respect to $\testFuncMap$ if $u \in W^{1,q}(\Omega)$ for some $q > N$ and $u$ is $\phi_{\intpt}$-Green-representable for every $\intpt \in \Omega$.
\end{definition}

As a direct consequence of Theorem \ref{thm:thm:Nd_representation_local} and the regularity discussion above, we obtain:

\begin{corollary}\label{cor:Nd_representation_global}
Let $\Omega \subset \mathbb{R}^N$ ($N \ge 2$) be a bounded $N$-dimensional polytopic domain and let $f \in L^p(\Omega)$ with $p > N/2$. 
If a solution $u$ of problem \eqref{eq:main} satisfies the regularity condition $u \in H^1_0(\Omega) \cap W^{1,q}(\Omega)$ for some $q > N$, then $u$ is globally Green-representable with respect to any mapping $\testFuncMap$ constructed by Construction \ref{con:test-function}.
\end{corollary}

This regularity condition $u \in W^{1,q}(\Omega)$ ($q>N$) also ensures, by Sobolev embedding, that $u$ is continuous up to the boundary, making the pointwise evaluation $u(\intpt)$ is well defined everywhere in $\overline{\Omega}$.

\subsubsection{Globally Green-representable solutions in special cases}\label{subsubsec:gr-existence-special-cases}
The representation \eqref{eq:green_rep} reveals a notable property: every globally Green-representable solution admits pointwise evaluation through a simple formula. While globally Green-representable solutions generally form a proper subset of weak solutions, the weak solutions of \eqref{eq:main} become globally Green-representable (with particularly convenient construction of $\phi_{\intpt}$ along with Construction \ref{con:test-function}) in the following two basic cases:
\begin{enumerate}
\item Solutions in one-dimensional domains
\item Solutions in smooth domains (e.g., $C^{1,1}$)
\end{enumerate}

For one-dimensional domains, the fundamental solution has the form $\fundsol(s,x) = -\frac{|x-s|}{2}$, and for any interior point $\intpt$, we can construct a test function $\phi_{\intpt} \in \testFuncSet$ as a linear combination of fundamental solutions such that $\phi_{\intpt} \in H^1_0(\Omega)$. This construction directly establishes \eqref{eq:green_rep}.
In other words, for any $\intpt \in \Omega$, the intersection of $\testFuncSet$ and $H^1_0(\Omega)$ is nonempty.
Since $\phi_{\intpt}$ vanishes on the boundary, the boundary integral term disappears, resulting in a formulation analogous to the standard weak form.
The detailed construction of such test functions is presented in Section \ref{sec:onedim}.

For smooth domains $\Omega$, solutions inherit higher regularity. If $f \in L^p(\Omega)$ with $p > N/2$,as in Assumption \ref{assum:f-integrability}, then standard elliptic regularity theory implies $u \in W^{2,p}(\Omega)$. By the Sobolev embedding theorem, we obtain $u \in W^{1,q}(\Omega)$ with $q > N$, satisfying the integrability requirement discussed above. Moreover, for such $u$, Green's formula yields the following identity
\begin{align*}
\langle \nabla u, \nabla \phi_{\intpt} \rangle &= -\int_\Omega \Delta u\,\phi_{\intpt} dx + \int_{\partial\Omega}\frac{\partial u}{\partial n}\phi_{\intpt} d\belem \\
&= \langle f,\phi_{\intpt} \rangle + \int_{\partial\Omega}\frac{\partial u}{\partial n}\phi_{\intpt} d\belem,
\end{align*}
where $-\Delta u = f$ was used in the last equality. This directly verifies \eqref{eq:green_rep}.
Notably, this verification holds for all test functions $\phi_{\intpt} \in \testFuncSet$ constructed via Construction \ref{con:test-function}.

\section{Green-representable sub- and super-solutions}\label{sec:generalized_upper_lower_solutions}
In this section, we introduce the notion of sub- and super-solutions within the Green-representable framework. Our goal is to enclose the actual solution of \eqref{eq:main} between a super-solution and a sub-solution over the entire domain. Accordingly, we focus on solutions that are globally Green-representable with respect to \(\testFuncMap\), i.e., those with the regularity \(u \in W^{1,q}(\Omega)\) for \(q > N\).
In Subsection~\ref{subsec:greensupersub-def}, we formalize the definitions of \(\testFuncMap\)-Green-representable sub- and super-solutions, which are essential for deriving rigorous solution bounds.
Subsection~\ref{subsec:comparison-thm} establishes a comparison theorem that guarantees the enclosure property.
Finally, Subsection~\ref{subsec:practical-remarks} discusses practical implications and the role of boundary conditions in our framework.

\subsection{Definitions of Green-representable sub- and super-solutions}\label{subsec:greensupersub-def}
We begin by formalizing the definitions of \(\testFuncMap\)-Green-representable sub- and super-solutions,
which are key to obtaining rigorous solution bounds.

\begin{definition}[Green-representable super-solution]\label{def:general_super}
A function $\overline{u} \in W^{1,q}(\Omega)$ $(q>N)$ is called a Green-representable super-solution of \eqref{eq:main} with respect to $\testFuncMap$ (or simply a \(\testFuncMap\)-Green-representable super-solution) if there exists a nonnegative constant $c$ and a mapping $\testFuncMap: \intpt \mapsto \phi_{\intpt}$ that assigns to each point $\intpt \in \Omega$ a test function $\phi_{\intpt}$ of the form \eqref{eq:test-function} (see Construction \ref{con:test-function}) such that
\begin{equation}
\label{eq:green_representable_super_ineq}
\left\langle\nabla\overline{u},\nabla\phi_{\intpt}\right\rangle
 \geq \langle f,\phi_{\intpt}\rangle + \int_{\partial\Omega}\overline{u}\frac{\partial\phi_{\intpt}}{\partial n}d\belem
\end{equation}
and
\begin{equation}
\label{eq:boundary_condition_super}
\overline{u} \geq 0 \quad \text{ on } \; \partial\Omega.
\end{equation}
\end{definition}
A \(\testFuncMap\)-Green-representable sub-solution $\underline{u}\in W^{1,q}(\Omega)$ $(q>N)$ is defined analogously by reversing all inequalities: there exists a mapping $\testFuncMap: \intpt \mapsto \phi_{\intpt}$ that assigns to each point $\intpt \in \Omega$ a test function $\phi_{\intpt}$ of the form \eqref{eq:test-function}
\begin{equation}
\label{eq:green_representable_sub_ineq}
\left\langle\nabla\underline{u},\nabla\phi_{\intpt}\right\rangle 
\leq \langle f,\phi_{\intpt}\rangle + \int_{\partial\Omega}\underline{u}\frac{\partial\phi_{\intpt}}{\partial n}d\belem
\end{equation}
and
\begin{equation}
\label{eq:boundary_condition_sub}
\underline{u} \leq 0 \quad \text{ on } \; \partial\Omega.
\end{equation}

Although the $W^{1,q}(q>N)$ regularity requirement for sub- and super-solutions may appear stronger than that of classical sub- and super-solutions defined in \eqref{eq:classical_super}, this is compensated for by restricting attention to a much smaller class of test functions. Specifically, the admissible test functions $\phi_{\intpt}$ are considerably more specialized than the test functions $\phi \in H^1_0(\Omega) \cap L^2_+(\Omega)$ used in \eqref{eq:classical_super}, since the range of $\testFuncMap$ is a strict subset of $H^1_0(\Omega)$. In the one-dimensional case, as shown later, this effectively reduces the problem to a single-variable setting.
This trade-off proves advantageous. As demonstrated in Section~\ref{sec:onedim}, the proposed framework enables the construction of sub- and super-solutions using piecewise linear functions, which is not possible within the classical framework.

Another key distinction from the classical definition \eqref{eq:classical_super} is the presence of the boundary integral term $\int_{\partial\Omega}\overline{u}\frac{\partial\phi_{\intpt}}{\partial n}d\belem$. In \eqref{eq:classical_super}, increasing the values of $\overline{u}$ on the boundary does not affect the defining inequality. In contrast, the present definition introduces a trade-off. Larger boundary values of $\overline{u}$ indicate a greater deviation from the true solution, but they also make the inequality \eqref{eq:green_representable_super_ineq} easier to satisfy, provided that $\frac{\partial\phi_{\intpt}}{\partial n} \leq 0$ on $\partial\Omega$.

\subsection{Comparison theorem}\label{subsec:comparison-thm}
We now establish a comparison theorem that characterizes the relationship between $\testFuncMap$-Green-representable sub- and super-solutions and the true solution $u$. The central observation underlying this theorem is that the true solution $u$ is itself $\testFuncMap$-Green-representable with respect to the same mapping $\testFuncMap$. This shared representability allows inequalities to be derived by comparing \eqref{eq:green_rep} with \eqref{eq:green_representable_super_ineq} and the corresponding sub-solution inequality.

\begin{theorem}\label{thm:comparison_h10}
Assume that there exists a mapping $\testFuncMap$ for which a solution $u \in H^1_0(\Omega)\cap W^{1,q}(\Omega)$ $(q>N)$ of \eqref{eq:main} is $\testFuncMap$-Green-representable, where $\phi_{\intpt}=\testFuncMap(\intpt)$.
For this same mapping $\testFuncMap$,
let $\overline{u}$ and $\underline{u}$ be a $\testFuncMap$-Green-representable super-solution and sub-solution of \eqref{eq:main}, respectively.
Then,
\begin{equation}\label{eq:comparison_ineq}
\underline{u} - \gamma_{\intpt} \leq u \leq \overline{u} + \gamma_{\intpt} \quad \mathrm{in} \; \Omega,
\end{equation}
where 
\begin{equation}\label{eq:gamma_def}
\gamma_{\intpt} := \frac{1}{\intcoef}\left\langle\frac{\partial u}{\partial n}, \gamma(\phi_{\intpt})\right\rangle_{H^{-1/2}, H^{1/2}}.
\end{equation}
\end{theorem}

\begin{proof}
We first recall a necessary identity.
From \eqref{eq:nabla_u_phi_coupling}, for any $\intpt \in \Omega$,
\begin{equation}
\intcoef\overline{u}(\intpt) = \langle \nabla \overline{u}, \nabla \phi_{\intpt} \rangle - \int_{\partial\Omega}\overline{u}\frac{\partial\phi_{\intpt}}{\partial n}d\belem.
\end{equation}
Using \eqref{eq:green_representable_super_ineq}, it follows that
\begin{align*}
\intcoef\overline{u}(\intpt) &= \langle \nabla \overline{u}, \nabla \phi_{\intpt} \rangle - \int_{\partial\Omega}\overline{u}\frac{\partial\phi_{\intpt}}{\partial n}d\belem \geq \langle f,\phi_{\intpt} \rangle,
\end{align*}
while from \eqref{eq:green_rep}
\begin{equation}
\intcoef u(\intpt) = \langle f,\phi_{\intpt} \rangle + \left\langle\frac{\partial u}{\partial n}, \gamma(\phi_{\intpt})\right\rangle_{H^{-1/2}, H^{1/2}}.
\end{equation}
Combining these relations yields,
\begin{align*}
\intcoef(\overline{u}(\intpt)-u(\intpt)) \geq -\left\langle\frac{\partial u}{\partial n}, \gamma(\phi_{\intpt})\right\rangle_{H^{-1/2}, H^{1/2}}.
\end{align*}
An analogous argument applies to the sub-solution case. Using \eqref{eq:green_representable_sub_ineq}, \eqref{eq:boundary_condition_sub}, and following the same arguments as above, we obtain the lower bound in \eqref{eq:comparison_ineq}.
Since the resulting inequality holds for every $\intpt \in \Omega$, the proof is complete.
\end{proof}

\subsection{Practical implications and boundary conditions}\label{subsec:practical-remarks}
The sharpness of the enclosure \eqref{eq:comparison_ineq} depends on the sign of \(\gamma_{\intpt}\) in \eqref{eq:gamma_def}. When \(\gamma_{\intpt} \leq 0\), the bounds reduce to \(\underline{u} \leq u \leq \overline{u}\), yielding a gap-free pointwise enclosure. The ideal case \(\gamma_{\intpt} = 0\) occurs when the test function \(\phi_{\intpt}\) vanishes on \(\partial\Omega\), thereby eliminating the boundary integral contribution. In practice, such an exact construction is feasible mainly in one-dimensional domains (see Section~\ref{sec:onedim}).
For higher-dimensional domains, determining the sign of \(\gamma_{\intpt}\) is more challenging because it involves the boundary integral \(\langle\frac{\partial u}{\partial n}, \gamma(\phi_{\intpt})\rangle_{H^{-1/2}, H^{1/2}}\). Classical tools such as Hopf's lemma require stronger smoothness assumptions on the solution and the domain, which may not hold in the present setting. Instead, we adopt a weak formulation approach that avoids imposing such regularity conditions.Consider the auxiliary problem of finding \(v \in H^1(\Omega)\) such that
\begin{equation}\label{eq:auxiliary_bvp}
\left\{\begin{array}{l}
(\nabla v,\nabla\psi) = 0 \quad \text{for all } \psi\in H_{0}^{1}(\Omega),\\
v = \gamma(\phi_{\intpt}) \quad \text{on } \partial\Omega.
\end{array}\right.
\end{equation}
By the Lax-Milgram theorem, this problem admits a unique solution since \(\gamma(\phi_{\intpt}) \in H^{1/2}(\partial\Omega)\). Taking \(\psi = u\) yields \((\nabla v, \nabla u) = 0\), and integration by parts with \(u = 0\) on \(\partial\Omega\) yields
\[
\left\langle\frac{\partial u}{\partial n}, \gamma(\phi_{\intpt})\right\rangle_{H^{-1/2}, H^{1/2}} = -\int_{\Omega} f v \,dx.
\]
If \(f \geq 0\) in \(\Omega\) and \(\gamma(\phi_{\intpt}) \geq 0\) on \(\partial\Omega\), the weak maximum principle ensures \(v \geq 0\) in \(\Omega\), hence \(\gamma_{\intpt} \leq 0\) when \(\intcoef > 0\). Although the condition \(f \ge 0\) may appear restrictive, a general \(f\) source term can be handled by decomposition:
\begin{equation} \label{eq:decomposition}
\left\{
\begin{aligned}
-\Delta u_1 &= f - \min(f, 0), \\
-\Delta u_2 &= \min(f, 0),
\end{aligned}
\right.
\quad
\text{so that }
u = u_1 + u_2.
\end{equation}
More generally, one may choose any \(g \geq 0\) such that \(f - g \geq 0\) and split the equation accordingly. This flexibility is useful in practice, as it allows the decomposition to be tailored to the structure of the source term. For the lower bound \(\underline{u} \leq u\), the sign condition is reversed, requiring \(\gamma(\phi_{\intpt}) \leq 0\) on \(\partial\Omega\). In practical computations, one typically constructs \(\phi_{\intpt}\) to minimize its boundary values and then shifts it appropriately to enforce the desired sign. The construction of suitable test functions is described in Subsection~\ref{subsec:test_function_construction}.

\section{Application to One-dimensional Problems}\label{sec:onedim}
In this section, we apply the theory developed in the previous sections to the one-dimensional case with $\Omega = (0,1)$. This choice entails no loss of generality, as any bounded interval can be transformed into this setting by scaling.
We first show that the solution representation theory from Section~\ref{sec:rep-green} recovers the classical Green's function in one dimension. We then specialize the sub- and super-solution framework to this setting, derive explicit characterizations that accommodate discontinuous functions---a case that is difficult to treat within the conventional variational approach.
In Subsection~\ref{subsec:1d-testfunc}, we construct test functions explicitly and establish the solution representation formula.
Subsection~\ref{subsec:1d-subsupersol} specializes the Green-representable sub- and super-solution framework to one dimension, providing explicit conditions.
Finally, in Subsection~\ref{subsec:1d-numerical}, we present numerical experiments that validate our theoretical results and demonstrate the effectiveness of the proposed algorithm for constructing solution enclosures.

\subsection{Solution representation and test function construction}\label{subsec:1d-testfunc}

For this domain, the fundamental solution takes the explicit form
\begin{equation}
\fundsol_s(x)=-\frac{|x-s|}{2}, \quad 0 < s < 1.
\end{equation}
In the one-dimensional setting, we adopt a simplified notation in which the evaluation point is denoted by $s$ rather than $s_{\intpt}$, as all necessary exterior points can be explicitly determined from a single interior point. Specifically, we can construct a test function $\phi_s$ that exactly satisfies the boundary conditions as a linear combination of $\fundsol_{s}(x)$, $\fundsol_{-s}(x)$, and $\fundsol_{2-s}(x)$
\begin{align}
\label{eq:phi_explicit}
\phi_s(x)&=2\left (\fundsol_s(x)-\frac{\fundsol_s(x)-\fundsol_{-s}(x)}{\fundsol_s(1)-\fundsol_{-s}(1)} \fundsol_s(1)-\frac{\fundsol_s(x)-\fundsol_{2-s}(x)}{\fundsol_s(0)-\fundsol_{2-s}(0)}\fundsol_s(0)\right )\nonumber\\
&=\frac{1}{s(1-s)}\fundsol_{s}(x)+a_{-s}\fundsol_{-s}(x)+a_{2-s}\fundsol_{2-s}(x).
\end{align}
Here, the coefficients are explicitly given by
\begin{equation*}
a_{-s} = 2\,\frac{\fundsol_s(1)}{\fundsol_s(1) - \fundsol_{-s}(1)}, \quad a_{2-s} = 2\,\frac{\fundsol_s(0)}{\fundsol_s(0) - \fundsol_{2-s}(0)}.
\end{equation*}
Following the notation in \eqref{eq:test-function} and \eqref{eq:harmonic-part}, this construction corresponds to choosing $\intcoef = \frac{1}{s(1-s)}$, $s_1 = -s$, $s_2 = 2-s$, $a_1 = a_{-s}$, $a_2 = a_{2-s}$, and $C = 0$. These parameters are uniquely determined by $s$ (where $s = \intpt$) to satisfy the homogeneous boundary conditions.

While this construction might appear rather technical, $\phi_s$ has a simple geometric interpretation. It is a piecewise linear function on $(0,s)$ and $(s,1)$ satisfying the interpolation conditions $\phi_s(s)=1$ and $\phi_s(0)=\phi_s(1)=0$.
Multiplying $\phi_s$ by $\intcoef^{-1} = s(1-s)$ recovers the classical one-dimensional Green's function
\begin{equation}
G_s(x) = \intcoef^{-1}\phi_s(x) = \begin{cases}
s(1-x), & x \geq s \\
x(1-s), & x < s.
\end{cases}
\end{equation}
This function $\phi_s$ therefore serves as a representation kernel for the solution value $u(s)$, effectively playing the role of a Green's function, as shown in the following theorem.

\begin{theorem}
\label{thm:one_dim}
For the mapping $\testFuncMap: s \mapsto \phi_{s}$ constructed by \eqref{eq:phi_explicit}, a solution $u \in H^1_0(\Omega)$ of \eqref{eq:main} is globally Green-representable with respect to $\testFuncMap$, and for all $s \in \Omega$, we have
\begin{equation}
\label{eq:one_dim_sol}
\intcoef u(s)=\langle f, \phi_s \rangle.
\end{equation}
\end{theorem}
\begin{proof}
The boundary integral term in \eqref{eq:green_rep} vanishes because $\phi_s$ satisfies the homogeneous boundary conditions by construction. 
Consequetly, \eqref{eq:one_dim_sol} follows immediately, establishing the $\testFuncMap$-Green-representability of $u$ in the sense of Definition \ref{def:green_representable_global}.
In the one-dimensional setting, the embedding $H^1(\Omega) \hookrightarrow W^{1,q}(\Omega)$ holds for all $q \geq 1$, so the solution $u$ automatically satisfies the $W^{1,q}$-regularity requirement.
\end{proof}
\begin{remark}
\label{rem:general_1d}
Because $W^{1,q}$-regularity holds for all $q\geq 1$, in one dimension, Theorem \ref{thm:one_dim} extends beyond solutions of \eqref{eq:main}. In fact, for any $u \in H^1(\Omega)$, applying \eqref{eq:coupling_def} to the test function $\phi_s$ constructed above yields 
\begin{equation*}
\intcoef u(s)=\langle u', \phi_s' \rangle+\phi_s'(0)u(0)-\phi_s'(1)u(1).
\end{equation*}
which provides a representation formula for general $H^1$ functions.
\end{remark}

\subsection{Green-representable sub- and super-solutions}\label{subsec:1d-subsupersol}
We now specialize the theory of sub- and super-solutions introduced in Definition~\ref{def:general_super} to the one-dimensional case. Throughout this subsection, we denote by $\testFuncMap: s \mapsto \phi_{s}$ the mapping constructed by \eqref{eq:phi_explicit}.

In one dimension, with $\Omega = (0,1)$, the Green-representable super-solution condition \eqref{eq:green_representable_super_ineq} becomes
\begin{equation}
\label{eq:general_1d_before}
\int_0^1 \overline{u}'(x)\phi_s'(x)\,dx \geq \int_0^1 f(x)\phi_s(x)\,dx + \overline{u}(1)\frac{\partial\phi_s}{\partial n}\Big|_{x=1} + \overline{u}(0)\frac{\partial\phi_s}{\partial n}\Big|_{x=0}.
\end{equation}
Since $\phi_s$ is piecewise linear with $\phi_s(s) = 1$ and $\phi_s(0) = \phi_s(1) = 0$, the derivatives are
\begin{align}
\phi_s'(0) &= \frac{\phi_s(s) - \phi_s(0)}{s - 0} = \frac{1 - 0}{s} = \frac{1}{s}, \\
\phi_s'(1) &= \frac{\phi_s(1) - \phi_s(s)}{1 - s} = \frac{0 - 1}{1 - s} = -\frac{1}{1-s}.
\end{align}
The outward normal derivatives are $\frac{\partial\phi_s}{\partial n}\big|_{x=0} = -\phi_s'(0)$ and $\frac{\partial\phi_s}{\partial n}\big|_{x=1} = \phi_s'(1)$. Consequently, the boundary contribution becomes
\begin{equation}
\overline{u}(1)\frac{\partial\phi_s}{\partial n}\Big|_{x=1} + \overline{u}(0)\frac{\partial\phi_s}{\partial n}\Big|_{x=0} = \overline{u}(1)\phi_s'(1) - \overline{u}(0)\phi_s'(0) = -\frac{\overline{u}(1)}{1-s} - \frac{\overline{u}(0)}{s}.
\end{equation}
For practical applications, we assume $\overline{u}(0) = \overline{u}(1) = c$. Then
\begin{equation}
-\frac{\overline{u}(1)}{1-s} - \frac{\overline{u}(0)}{s} = -\frac{c}{1-s} - \frac{c}{s} = -c\intcoef.
\end{equation}
This observation allows us to relate the general boundary condition \eqref{eq:boundary_condition_super} to our one-dimensional formulation. Combining this result with \eqref{eq:general_1d_before}, we obtain
\begin{equation}
\int_0^1 \overline{u}'(x)\phi_s'(x)\,dx \geq \int_0^1 f(x)\phi_s(x)\,dx - c\intcoef,
\end{equation}
which is equivalent to \eqref{eq:SS1} below.

By the one-dimensional Sobolev embedding, any function in $H^1(\Omega)$ belongs to $W^{1,q}(\Omega)$ for all $q \geq 1$. Consequently, the conditions for Green-representable sub- and super-solutions take an explicit form in one-dimensional case.

\begin{corollary}[One-dimensional characterization]\label{cor:1d_characterization}
Let $\Omega = (0,1)$ and let $\testFuncMap: s \mapsto \phi_s$ be the mapping constructed by \eqref{eq:phi_explicit}.
\begin{enumerate}
\item A function $\overline{u} \in H^1(\Omega)$ is a $\testFuncMap$-Green-representable super-solution if there exists a nonnegative constant $c$ such that
\begin{equation}
\label{eq:SS1}
\int_0^1 \overline{u}'(x)\phi_s'(x)\,dx + c\intcoef \geq \int_0^1 f(x)\phi_s(x)\,dx \quad \text{ for all } \; s\in (0,1)
\end{equation}
and
\begin{equation}
\label{eq:SS2}
\overline{u} - c \geq 0 \quad \text{ on } \; \partial\Omega.
\end{equation}

\item A function $\underline{u} \in H^1(\Omega)$ is a $\testFuncMap$-Green-representable sub-solution if there exists a nonnegative constant $c$ such that
\begin{equation}
\label{eq:SS3}
\int_0^1 \underline{u}'(x)\phi_s'(x)\,dx - c\intcoef \leq \int_0^1 f(x)\phi_s(x)\,dx \quad \text{ for all } \; s\in (0,1)
\end{equation}
and
\begin{equation}
\label{eq:SS4}
\underline{u} + c \leq 0 \quad \text{ on } \; \partial\Omega.
\end{equation}
\end{enumerate}
\end{corollary}

The constant $c$ acts as a boundary shift parameter: condition \eqref{eq:SS2} is a direct translation of the general boundary condition \eqref{eq:boundary_condition_super}, while the term $c\intcoef$ in \eqref{eq:SS1} arises from the explicit evaluation of the boundary integral. A key advantage of this formulation is that $\phi_s$ satisfies homogeneous boundary conditions, so $\gamma_s = 0$ in Theorem~\ref{thm:comparison_h10}, yielding a particularly clean comparison result.

\begin{remark}[Optimality of Constant Sub- and Super-solutions]
\label{rem:optimal}
Recall that $\intcoef$ depends on $s(=\intpt)$ through the construction of $\phi_s$ given in \eqref{eq:phi_explicit}. We consider constant super-solutions of the form $\overline{u}=c$. From \eqref{eq:SS1}, such a constant $c$ must satisfy 
\begin{equation*}
c\intcoef(s) \geq \langle f,\phi_s \rangle \quad \text{ for all } s\in \Omega.
\end{equation*}
This observation motivates the choice
\begin{equation}
\label{eq:optimal_c}
c = \sup_{s\in \Omega}\frac{\langle f,\phi_s \rangle}{\intcoef(s)}
\end{equation}
which yields the smallest admissible constant super-solution. On the other hand, by Theorem~\ref{thm:one_dim}, the solution $u$ satisfies
\begin{equation*}
\intcoef(s)u(s) = \langle f,\phi_s \rangle \quad \text{ for all } s\in \Omega,
\end{equation*}
and therefore
\begin{equation*}
\sup_{s\in \Omega}\frac{\langle f,\phi_s \rangle}{\intcoef(s)} = \sup_{s\in \Omega}u(s).
\end{equation*}
Consequently, the constant function $\overline{u} = \displaystyle\sup_{s\in\Omega} u(s)$ satisfies the super-solution condition \eqref{eq:SS1}. Moreover, any constant smaller than this value fails to satisfy \eqref{eq:SS1} for some $s \in \Omega$; that is, it is the optimal (i.e., the smallest) constant super-solution.
To illustrate this optimality, consider the case $f(x)=A>0$ for all $x\in \Omega$. A direct calculation shows that the exact solution is $u(x)=Ax(1-x)/2$ with maximum value $u(1/2)=A/8$. The bound given by \eqref{eq:optimal_c} recovers this value exactly, demonstrating the sharpness of our approach even for simple examples.

By an analogous argument, the constant function $\underline{u} = \displaystyle\inf_{s\in\Omega} u(s)$ is the optimal (i.e., largest) constant sub-solution. It satisfies \eqref{eq:SS3}, and any constant larger than this value violates the sub-solution condition.
\end{remark}

\subsection{Numerical experiments}\label{subsec:1d-numerical}

In this subsection, we present numerical experiments to evaluate the performance of the theoretical framework and to demonstrate the effectiveness of the proposed method for constructing sub- and super-solutions. We focus on the domain $\Omega = (0, 1)$ and examine how the accuracy of the solution enclosure depends on different source terms $f$ and parameter choices. These experiments provide insight into the practical implementation of our theoretical results and verify the sharpness of the error bounds we derive.

We begin by introducing an algorithm for constructing Green-representable sub- and super-solutions. The algorithm is based on a finite difference discretization of the Poisson equation, with particular attention to ensuring that the resulting functions satisfy the conditions in Definition \ref{def:general_super} (specifically, the inequalities \eqref{eq:SS1} and \eqref{eq:SS2}).

Before presenting the algorithm, we introduce the key notation and parameters used throughout this subsection:
\begin{itemize}
    \item $I_i = [x_i, x_{i+1}]$ denotes the $i$-th subinterval of the mesh, where $x_i$ and $x_{i+1}$ are consecutive mesh points. The super-solution condition is tested within this interval.
    \item $\overline{f}_i$ denotes the discretized right-hand side at mesh point $x_i$. When the super-solution condition fails in the interval $I_i$, both $\overline{f}_i$ (at the left endpoint) and $\overline{f}_{i+1}$ (at the right endpoint) are incremented by $\varepsilon$.
    \item $A$ denotes the discretized Laplacian matrix for interior points, and $A^{-1}$ refers to its inverse operator, which can be implemented using either finite difference or finite element methods. In our implementation, we use a simple one-dimensional finite difference method.
    \item $c$ is the boundary shift constant introduced in Definition \ref{def:general_super}, which plays a central role in balancing the tightness of the solution enclosure and the satisfaction of boundary conditions.
    \item $\varepsilon$ denotes a small perturbation parameter that controls the local adjustment of the discretized right-hand side. In our implementation, we set $\varepsilon = O(h) \cdot |f|$, where $|f|$ denotes the magnitude of the source term. This scaling ensures that the perturbation is proportional to both the mesh size and the source term.
\end{itemize}

\begin{algorithm}
\caption{Construction of Green-representable Super-solutions}
\label{alg:super_solution}
\begin{algorithmic}[1]
    \State Initialize $\overline{\mathbf{u}} \gets (c, A^{-1}\overline{\mathbf{f}} + c, c)$, where $A$ is the discretized Laplacian matrix for interior points
    \Repeat
        \State $\mathrm{updated} \gets \mathrm{false}$
        \For{$i=1,2,\ldots,n$}
            \If{$\langle\overline{u}',\phi_{s}'\rangle+c\intcoef<\langle f,\phi_{s}\rangle$ for some $s \in I_{i}$}
                \State $\overline{f}_{i} \gets \overline{f}_{i}+\varepsilon$
                \State $\overline{f}_{i+1} \gets \overline{f}_{i+1}+\varepsilon$
                \State $\mathrm{updated} \gets \mathrm{true}$
            \EndIf
        \EndFor
        \If{$\mathrm{updated}$}
            \State $\overline{\mathbf{u}} \gets (c, A^{-1}\overline{\mathbf{f}} + c, c)$
        \EndIf
    \Until{not $\mathrm{updated}$}
    \State \Return $\overline{\mathbf{u}}$
\end{algorithmic}
\end{algorithm}

To construct sub-solutions $\underline{u}$, we follow the same algorithm but reverse the inequality check in line 3 and appropriately adjust the signs in the subsequent updates accordingly.

To clarify the mathematical justification for this algorithm, we examine the local update performed in lines 5--7.
When the condition fails on the interval $I_i = [x_i, x_{i+1}]$ with width $h$, the algorithm increments the discrete source values $\overline{f}_i$ and $\overline{f}_{i+1}$ by $\varepsilon$.
In the continuous setting, this operation corresponds to adding a perturbation $\delta f(x)$ to the source term. This perturbation can be modeled as a localized function supported on $I_i$ with magnitude approximately $\varepsilon$:
\begin{equation}
    \delta f(x) \approx \begin{cases} \varepsilon & \text{if } x \in I_i, \\ 0 & \text{otherwise.} \end{cases}
\end{equation}
Let $\delta u$ denote the resulting correction to the solution, which satisfies $-\Delta (\delta u) = \delta f$ with homogeneous Dirichlet boundary conditions.
Using the Green's function $G(x,y)$ for the operator $-\Delta$, the correction $\delta u$ admits the representation
\begin{equation}
    \delta u(x) = \int_0^1 G(x,y) \delta f(y) \, dy \approx \int_{I_i} G(x,y) \varepsilon \, dy = \varepsilon h G(x,s) \quad \text{for some } s \in I_i,
\end{equation}
where the final equality follows from the mean value theorem for integrals.

This estimate shows that the correction is of order $O(\varepsilon h)$. 
Consequently, choosing $\varepsilon = O(h)$ (specifically $\varepsilon \propto h|f|$ in our implementation) yields a total correction of order $O(h^2)$.
This scaling matches the $O(h^2)$ discretization error of the finite difference scheme. As a result, the verification procedure introduces a perturbation that is comparable to the truncation error, allowing for tight solution enclosures without degrading the overall convergence rate.

In the following subsections, we present numerical results for several test cases and examine the relationship between the error bounds, the mesh size $h$, and the boundary shift parameter $c$. We consider both constant and discontinuous source terms, demonstrating the flexibility of the proposed approach across different problem settings.

\subsubsection{Case of Constant Source Term}
We first examine the case where the source term $f$ is constant throughout the domain. This simple setting allows us to analyze the fundamental behavior of the proposed method before addressing more complex problems. We particularly focus on identifying suitable choices for the key parameters: the boundary shift parameter $c$ in Corollary \ref{cor:1d_characterization} (specifically, conditions \eqref{eq:SS2} and \eqref{eq:SS4}) and the perturbation parameter $\varepsilon$ in Algorithm \ref{alg:super_solution}. Because the source term is smooth in this case, the one-dimensional finite difference method achieves $O(h^2)$ accuracy, which suggests that $c$ should be proportional to $h^2$ for optimal results. For $\varepsilon$, we set $\varepsilon = 0.25h \cdot |f|$ in our implementation, where $|f|$ is the magnitude of the source term.

Figure \ref{fig:error_vs_c_h_f1_f5} shows how the error (the maximum difference between sub- and super-solutions) varies with the mesh size $h$ for different values of $c$, with $\varepsilon = 0.25h \cdot |f|$. The left and right panels correspond to $f=1$ and $f=5$, respectively.

\begin{figure}[H]
    \begin{minipage}{0.5\linewidth}
        \centering
        \includegraphics{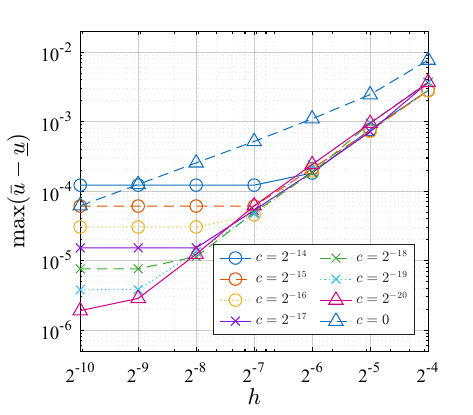}
    \end{minipage}%
    \begin{minipage}{0.5\linewidth}
        \centering
        \includegraphics{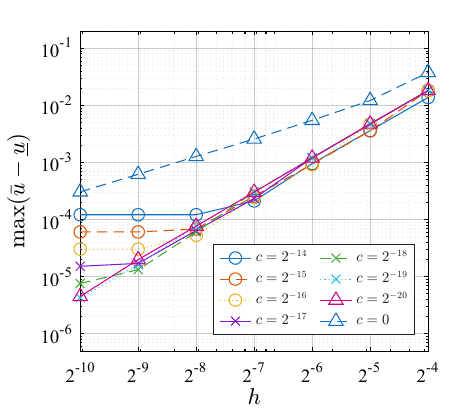}
    \end{minipage}
    \caption{Enclosure error versus mesh size $h$ for various $c$. Left: $f=1$, right: $f=5$.}\label{fig:error_vs_c_h_f1_f5}
\end{figure}

The results show that choosing $c=0$ leads to substantially larger enclosure errors, confirming the importance of the boundary shift parameter. Table \ref{tab:optimal_c_values} lists the optimal values of $c$ for various mesh sizes, demonstrating a consistent $O(h^2)$ scaling behavior.

\begin{table}[H]
\centering
\caption{Optimal values of boundary shift parameter $c$ for various mesh sizes.}
\label{tab:optimal_c_values}
\begin{tabular}{ccc}
\hline
Mesh size $h$ & Optimal $c$ for $f=1$ & Optimal $c$ for $f=5$ \\
\hline
$2^{-9}$ & $2^{-20}$ & $2^{-18}$ \\
$2^{-8}$ & $2^{-18}$ & $2^{-16}$ \\
$2^{-7}$ & $2^{-16}$ & $2^{-14}$ \\
$2^{-6}$ & $2^{-14}$ & -- \\
\hline
\end{tabular}
\end{table}

To determine the optimal proportionality constant, Figure \ref{fig:error_vs_C_f1_f5} shows the error behavior for different values of $C = c/h^2$ as the mesh size $h$ varies. Based on these results, we observe that $C = 0.2$ for $f=1$ and $C = 1$ for $f=5$ produce the smallest enclosure errors. This observation suggests that a practical choice is $C = 0.2|f|$, although the exact optimal value may depend on the specific problem.

\begin{figure}[H]
    \begin{minipage}{0.5\linewidth}
        \centering
        \includegraphics{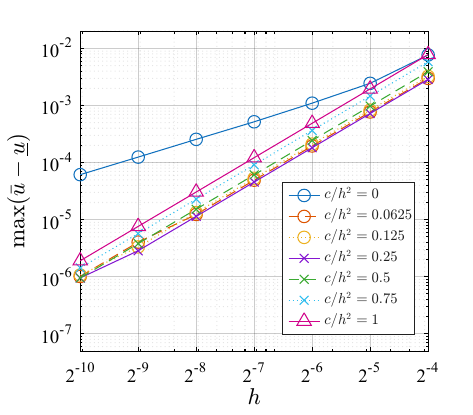}
    \end{minipage}%
    \begin{minipage}{0.5\linewidth}
        \centering
        \includegraphics{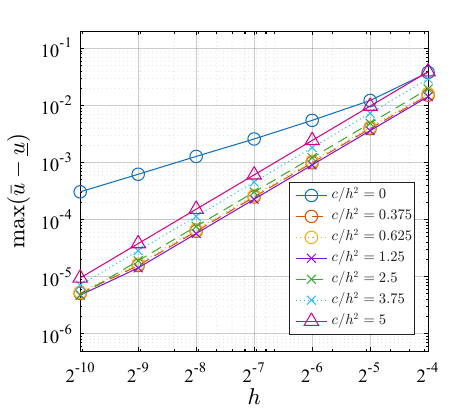}
    \end{minipage}
    \caption{Error behavior as a function of mesh size $h$ for different values of $C = c/h^2$. Left: $f=1$, right: $f=5$.}\label{fig:error_vs_C_f1_f5}
\end{figure}

Importantly, the numerical experiments indicate that the choice of $C = c/h^2$ is not overly sensitive. Several values of $C$ in the vicinity of the optimal value yield comparable error bounds. This robustness is a favorable property of the proposed method, as it enables practical implementation without requiring highly precise parameter tuning.

Figure \ref{fig:enclosure_visualization} illustrates the solution enclosures produced $u$ by the super-solution $\overline{u}$ and the sub-solution $\underline{u}$ constructed using Algorithm \ref{alg:super_solution} for $f=1$. The left panel corresponds to a coarser mesh with ($h=2^{-3}$), while the right panel shows the result for a finer mesh with ($h=2^{-4}$). In both cases, the exact solution is strictly enclosed between the upper and lower bounds over the entire domain. Moreover, the gap between the bounds decreases as the mesh is refined, visually demonstrating convergence of the enclosure.

\begin{figure}[htbp]
    \centering
    \begin{minipage}{0.4\linewidth}
        \centering
        \includegraphics[width=\linewidth]{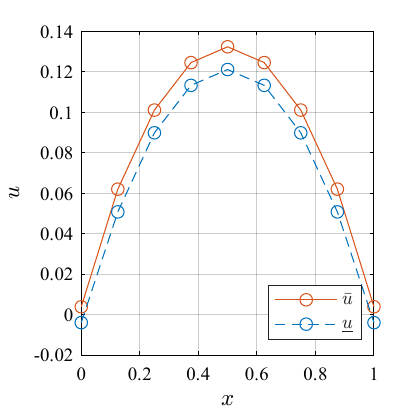}
        \subcaption{$h=2^{-3}$}
        \label{fig:encl8}
    \end{minipage}
    \begin{minipage}{0.4\linewidth}
        \centering
        \includegraphics[width=\linewidth]{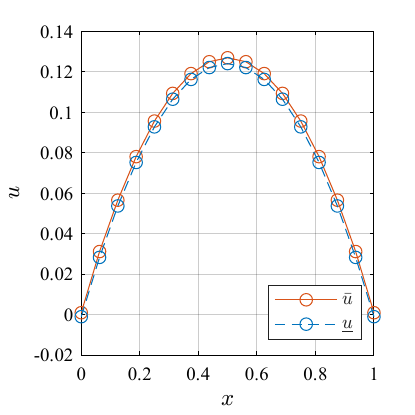}
        \subcaption{$h=2^{-4}$}
        \label{fig:encl16}
    \end{minipage}
    \caption{Visualization of solution enclosures for the constant source term $f=1$. The exact solution is tightly bound by the computed super- and sub-solutions.}
    \label{fig:enclosure_visualization}
\end{figure}

\subsubsection*{Optimal Bounds for Constant Sub- and Super-solutions}
We now present numerical evidence for the theoretical optimality discussed in Remark \ref{rem:optimal}. Figure \ref{fig:optimal_bounds} compares the exact solution $u$ with the optimal constant super-solution $\overline{c}_{\text{opt}}$ and the optimal constant sub-solution $\underline{c}_{\text{opt}}$ for $f=1$ and $f=5$. The computed values $\overline{c}_{\text{opt}} = 0.125$ for $f=1$ and $\overline{c}_{\text{opt}} = 0.625$ for $f=5$ precisely match the theoretical prediction $\overline{c}_{\text{opt}} = \sup\{u(s) : s\in\Omega\}$. Analogously, the optimal sub-solutions coincide with $\inf\{u(s) : s\in\Omega\}$, confirming the sharpness of the derived bounds.
These results demonstrate that the proposed framework attains the tightest possible constant closures, thereby validating both the theoretical analysis and its numerical realization.

\begin{figure}[H]
    \begin{minipage}{0.5\linewidth}
        \centering
        \includegraphics{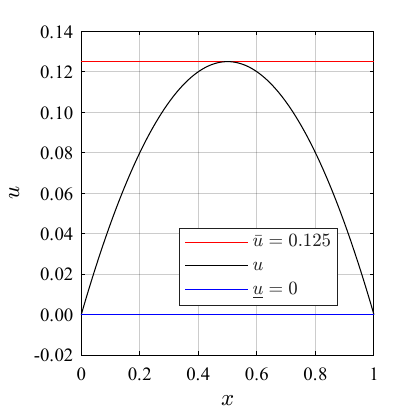}
    \end{minipage}%
    \begin{minipage}{0.5\linewidth}
        \centering
        \includegraphics{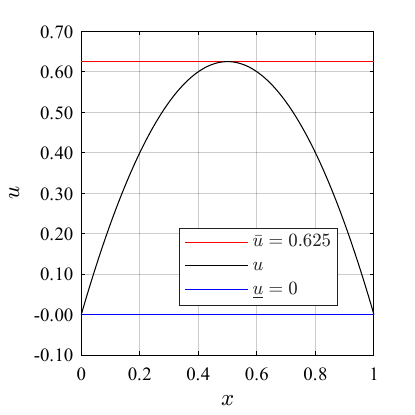}
    \end{minipage}
    \caption{Exact solution $u$ and optimal constant sub-/super-solutions. Left: $f=1$, right: $f=5$.}
    \label{fig:optimal_bounds}
\end{figure}

\subsubsection{Case of Discontinuous Source Terms}

Having established the behavior of the proposed method for constant source terms, we now examine its performance for discontinuous source terms. Such discontinuities reduce the regularity of the exact solution, which makes the construction of sharp enclosures more challenging. Nevertheless, such problems are of practical importance in many applications.

We consider piecewise constant source terms with a single discontinuity at a point $a \in (0,1)$:
\begin{equation}
f(x) = 
\begin{cases}
1, & x < a \\
1 + 2^{-5}n, & x \geq a,
\end{cases}
\end{equation}
Where the parameter $n$ controls the magnitude of the jump. Figure \ref{fig:discontinuous_approx} shows the approximate solutions for $a=0.25$ and $a=0.50$ with $n=4$. We systematically vary both the discontinuity location $a$ and the mesh size $h$, and we apply Algorithm~\ref{alg:super_solution} to construct the corresponding solution enclosures. Figures~\ref{fig:dchc_025} and~\ref{fig:dchc_050} present the enclosure error as a function of $h$ for several values of the boundary shift parameter $c$ and for $n \in \{1, 2, 3, 4\}$.

\begin{figure}[htbp]
    \centering
    \begin{minipage}{0.4\linewidth}
        \centering
        \includegraphics[width=\linewidth]{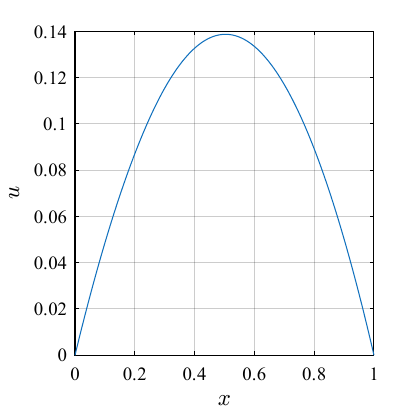}
        \subcaption{$a=0.25$}
        \label{fig:approx025}
    \end{minipage}
    \begin{minipage}{0.4\linewidth}
        \centering
        \includegraphics[width=\linewidth]{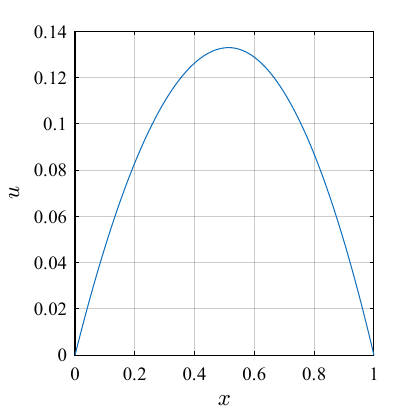}
        \subcaption{$a=0.50$}
        \label{fig:approx050}
    \end{minipage}
    \caption{Approximate solutions for discontinuous source terms. The discontinuity in $f$ induces a kink in the solution, which is rigorously enclosed by the proposed method.}
    \label{fig:discontinuous_approx}
\end{figure}

\begin{figure}[H]
    \centering
    \includegraphics{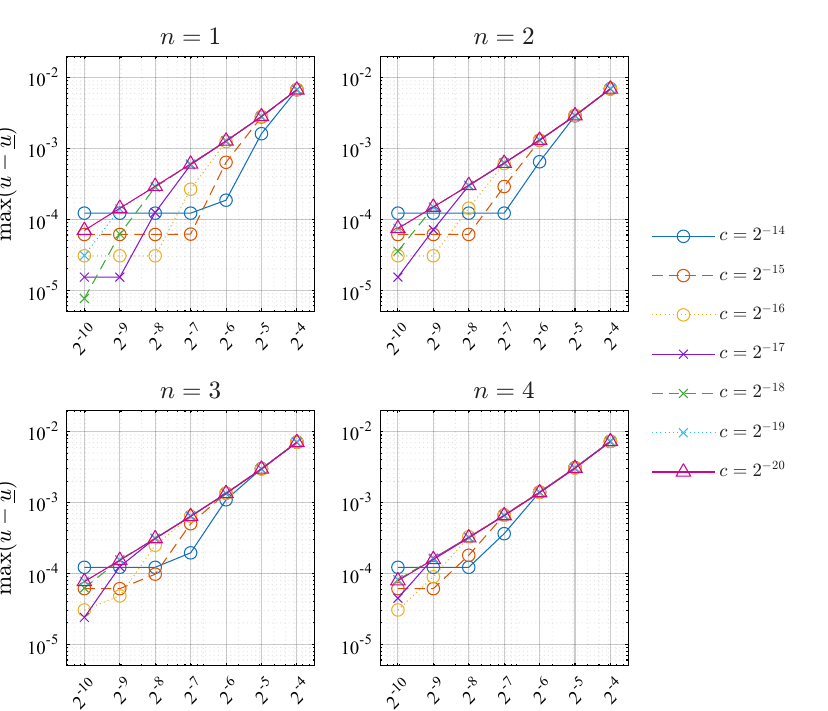}
    \caption{Error as a function of mesh size $h$ for various boundary shift parameters $c$. The source term has a discontinuity at $a = 0.25$, taking the values $f(x) = 1$ for $x < a$ and $f(x) = 1 + 2^{-5}n$ for $x \geq a$. Panels correspond to different values of $n \in \{1, 2, 3, 4\}$.}
    \label{fig:dchc_025}
\end{figure}

\begin{figure}[H]
    \centering
    \includegraphics{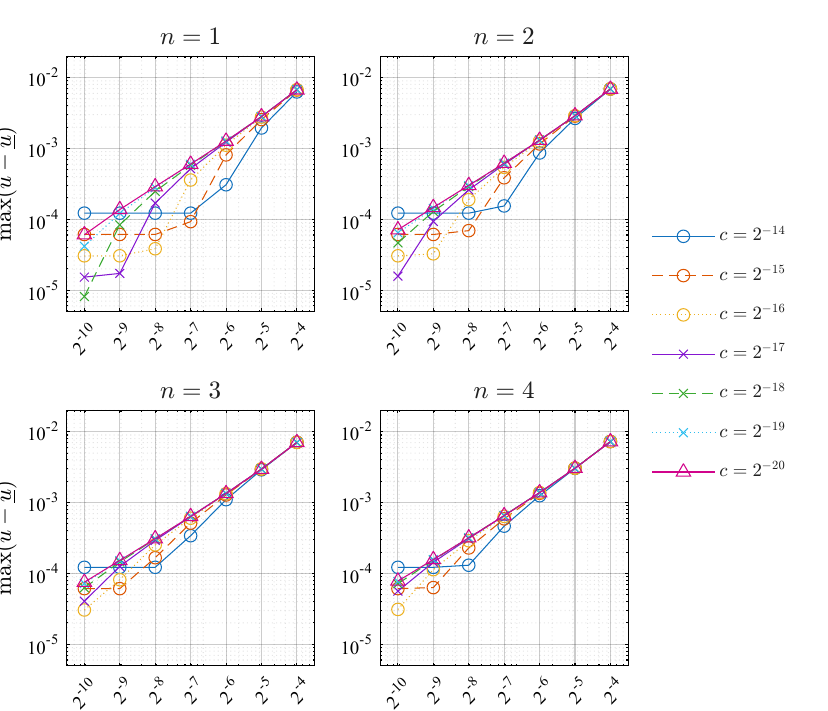}
    \caption{Error as a function of mesh size $h$ for various boundary shift parameters $c$. The source term has a discontinuity at $a = 0.5$, taking the values $f(x) = 1$ for $x < a$ and $f(x) = 1 + 2^{-5}n$ for $x \geq a$. Panels correspond to different values of $n \in \{1, 2, 3, 4\}$.}
    \label{fig:dchc_050}
\end{figure}

Several important observations emerge from these experiments:

\begin{enumerate}
    \item For all discontinuity locations and jump magnitudes considered, the error decreases as the mesh size $h$ decreases. This behavior confirms convergence of the proposed method even in the presence of discontinuous source terms.
    
    \item In contrast to the case of constant source terms, the optimal boundary shift parameter $c$ for discontinuous source terms exhibits an approximately linear dependence on the mesh size $h$. Analysis of the error curves indicates that the minimum error is achieved for $c = O(h^{1.1})$, which differs markedly from the quadratic scaling observed in the smooth case. This reduction in convergence order is consistent with standard finite difference theory \cite{JovanovicSuli2014}, as reduced regularity at the discontinuity can degrade the global convergence rate to $O(h)$.
    
    \item As the magnitude of the jump increases (higher values of $n$), the absolute error increases as expected. However, the relationship governing the optimal parameters ($c \approx O(h^{1.1})$) remains consistent across all tested values of $n$. This observation indicates that the proposed method maintains robustness with respect to the size of the discontinuity, although the convergence behavior differs from that observed in the continuous case.
\end{enumerate}

These results demonstrate that the proposed framework provides reliable solution enclosures even for problems with discontinuous source terms, albeit with reduced convergence rates compared to problems with smooth data. This capability remains valuable in practical applications where discontinuities frequently arise, including multi-material problems, phase transitions, and shock waves.

Finite element methods can achieve $O(h^2)$ convergence rates for piecewise constant source terms when the discontinuity aligns with element boundaries. In contrast, for finite difference methods, specialized techniques such as the Immersed Interface Method (IIM) \cite{LeVequeLi1994} are typically required to obtain comparable accuracy. Our experiments employ a standard finite difference discretization without such enhancements and yield convergence of order $O(h^{1.1})$, which is consistent with classical theory for problems exhibiting reduced regularity. Notably, these results demonstrate that the enclosure framework does not rely on specialized discretization techniques and remains effective when combined with conventional numerical methods.

\section{Application to Two-dimensional Polygonal Domains}\label{sec:twodim}
In this section, we extend the proposed theory to two-dimensional domains, with particular emphasis on non-convex polygonal geometries. We begin in Subsection~\ref{subsec:2d-regularity} by examining solution regularity near non-convex corners and establishing Green-representability for polygonal domains.
Subsection~\ref{subsec:enclosure-theorems} derives the enclosure results that provide pointwise bounds and comparison.
Subsection~\ref{subsec:test_function_construction} describes the construction of test functions using the MFS.
In Subsection~\ref{subsec:numericalexample}, we present numerical examples demonstrating the effectiveness of our approach for both convex and non-convex domains.
Finally, Subsection~\ref{subsec:singular_integrals} addresses the rigorous evaluation of singular integrals involving logarithmic potentials, which is essential for verified computation.

\subsection{Regularity and Green-representability}\label{subsec:2d-regularity}
The applicability of the global Green-representability framework (Corollary \ref{cor:Nd_representation_global}) depends on the solution regularity $u \in W^{1,q}(\Omega)$ for some $q > N$. In the two-dimensional case ($N=2$), this condition becomes $q > 2$.
While standard elliptic regularity theory guarantees $H^2$ regularity (and thus $W^{1,q}$ for all finite $q$), in convex domains, solutions posed on non-convex polygonal domains may lose this regularity because of corner singularities.
In this subsection, we show that even in the presence of reentrant corners, the solution retains sufficient regularity (specifically $u \in W^{1,q}(\Omega)$ for some $q > 2$) to validate our representation formula.

Let $(r,\theta)$ be polar coordinates centered at a vertex with interior angle $\alpha \in (\pi, 2\pi)$.
According to Grisvard \cite{grisvard2011elliptic}, the solution $u$ in a neighborhood of such a reentrant corner admits the decomposition
\begin{equation}\label{eq:corner_decomp}
u(r,\theta) = \lambda u_{\text{sing}}(r,\theta) + v(r,\theta), \quad u_{\text{sing}}(r,\theta) := r^{\pi/\alpha}\sin(\pi\theta/\alpha),
\end{equation}
where $\lambda$ denotes the stress intensity factor and $v \in W^{2,p}(\Omega)$ represents the regular component of the solution. Here, $p > 1$ corresponds to the integrability of the source term $f \in L^{p}(\Omega)$.
Considering the Sobolev embedding for $v$, we note that the following embedding holds
\begin{equation*}
W^{2,p}(\Omega) \hookrightarrow W^{1,q}(\Omega)
\end{equation*}
where $q$ satisfies
\[
q = \begin{cases}
\dfrac{2p}{2-p}, & 1 < p < 2, \\[6pt]
\text{any finite }q < \infty, & p = 2, \\[6pt]
\infty, & p > 2.
\end{cases}
\]
As a consequence, since the function $2p/(2-p)$ is monotonically increasing for $1<p<2$, we conclude that $v \in W^{1,q}(\Omega)$ for some $q > 2$.

With this decomposition, for any small $\delta > 0$, we have $u_{\text{sing}} \in W^{1,\frac{2\alpha}{\alpha-\pi}-\delta}(\Omega)$. In particular, as the interior angle $\alpha$ approaches $2\pi$ corresponding to the most severe non-convex configuration, the upper bound on the admissible exponent approaches 4. Therefore, when $\Omega$ is a two-dimensional polygonal domain, the solution $u \in H^1_0(\Omega)$ of \eqref{eq:main} enjoys enhanced regularity in the sense that: $u \in W^{1,q}(\Omega)$ for some $q > 2 = N$. Since this regularity satisfies the assumption of Corollary \ref{cor:Nd_representation_global}, the Green-representability result applies directly to non-convex polygonal domains in two dimensions.

\begin{theorem}\label{thm:2d_representation}
Let $\Omega \subset \mathbb{R}^2$ be a bounded polygonal domain that may be non-convex, and let $f \in L^p(\Omega)$ with $p > 1$. 
Then a solution $u \in H^1_0(\Omega)$ of problem \eqref{eq:main} is globally Green-representable with respect to $\testFuncMap$ for any mapping $\testFuncMap: \intpt \mapsto \phi_{\intpt}$ constructed by Construction \ref{con:test-function}.
\end{theorem}

\subsection{Enclosure Results}\label{subsec:enclosure-theorems}
Combining Theorem \ref{thm:2d_representation} with the decomposition \eqref{eq:corner_decomp}, we obtain the following pointwise representation result.

\begin{theorem}\label{thm:point_estimate}
Let $\Omega \subset \mathbb{R}^2$ be a bounded polygonal domain (possibly non-convex) and let $f \in L^p(\Omega)$ with $p > 1$. For the solution $u \in H^1_0(\Omega)$ of problem \eqref{eq:main}, we have
\begin{equation}\label{eq:u_s0_solution}
\intcoef u(\intpt)=\langle f,\phi_{\intpt} \rangle+\int_{\partial\Omega}\frac{\partial u}{\partial n}\phi_{\intpt}d\belem
\end{equation}
for any fixed $\intpt \in \Omega$, where $\phi_{\intpt}$ is arbitrarily constructed according to Construction \ref{con:test-function}.
\end{theorem}

\begin{proof}
By Theorem \ref{thm:2d_representation}, the solution $u$ is globally Green-representable, so
\[
\intcoef u(\intpt) = \langle f, \phi_{\intpt} \rangle + \Big\langle \frac{\partial u}{\partial n}, \gamma(\phi_{\intpt}) \Big\rangle_{H^{-1/2}(\partial\Omega), H^{1/2}(\partial\Omega)}.
\]
It remains to verify that the duality pairing on the right-hand side coincides with the boundary integral $\int_{\partial\Omega} \frac{\partial u}{\partial n} \phi_{\intpt} \, d\belem$.

Using the decomposition \eqref{eq:corner_decomp}, we write $u = \lambda u_{\text{sing}} + v$ where $v \in W^{2,p}(\Omega)$.
\begin{itemize}
\item \textbf{Regular part:} Since $v \in W^{2,p}(\Omega)$, the trace theorem implies $\frac{\partial v}{\partial n} \in L^p(\partial\Omega) \subset L^1(\partial\Omega)$.
\item \textbf{Singular part:} Near a corner with interior angle $\alpha \in (\pi, 2\pi)$, the gradient of $u_{\text{sing}} = r^{\pi/\alpha} \sin(\pi\theta/\alpha)$ satisfies $|\nabla u_{\text{sing}}| \sim r^{\pi/\alpha - 1}$. On the boundary edges emanating from the corner, the normal derivative inherits this behavior. For this to be integrable, we require $\int_0^\epsilon r^{\pi/\alpha - 1} \, dr < \infty$, which holds if and only if $\pi/\alpha - 1 > -1$. Since $\alpha < 2\pi$, we have $\pi/\alpha > 1/2$, hence $\pi/\alpha - 1 > -1/2 > -1$. Thus $\frac{\partial u_{\text{sing}}}{\partial n} \in L^1(\partial\Omega)$.
\end{itemize}
Consequently, $\frac{\partial u}{\partial n} = \lambda \frac{\partial u_{\text{sing}}}{\partial n} + \frac{\partial v}{\partial n} \in L^1(\partial\Omega)$. Since $\phi_{\intpt}$ is smooth on $\partial\Omega$ (as $\intpt \in \Omega$ is an interior point), the duality pairing reduces to the standard boundary integral.
\end{proof}

Recalling the discussion in Subsection \ref{subsec:practical-remarks}, we have already established, via a weak formulation approach (specifically, utilizing the auxiliary problem \eqref{eq:auxiliary_bvp}), that the boundary coupling term is non-positive when $f \geq 0$ and the test function is non-negative on the boundary, and non-negative when $f \geq 0$ and the test function is non-positive on the boundary.
This observation leads to the following practical enclosure result.

\begin{corollary}\label{cor:enclosure}
With the same notation as in Theorem \ref{thm:point_estimate}, suppose that $f \geq 0$ in $\Omega$. 
Let $\overline{\phi_{\intpt}}$ and $\underline{\phi_{\intpt}}$ be test functions constructed according to Construction \ref{con:test-function} that satisfy $\overline{\phi_{\intpt}} \geq 0$ and $\underline{\phi_{\intpt}} \leq 0$ on $\partial\Omega$.
Then, for any $\intpt \in \Omega$, we have
\begin{equation}
\label{eq:enclosure}
\frac{1}{\intcoef}\langle f,\underline{\phi_{\intpt}} \rangle \leq u(\intpt) \leq \frac{1}{\intcoef}\langle f,\overline{\phi_{\intpt}} \rangle.
\end{equation}
\end{corollary}

This corollary is particularly significant from a practical perspective because it provides upper and lower bounds for the solution even when the test functions do not satisfy the boundary conditions. This is a notable feature of the proposed framework. Moreover, by selecting $\intpt$ arbitrarily within the domain, one can obtain localized and precise enclosures of the solution values at specific points of interest. This flexibility in pointwise evaluation represents another distinctive characteristic of the method.
The numerical examples in Subsection \ref{subsec:numericalexample} are based on this corollary.

Although the condition $f \geq 0$ may appear restrictive, as noted in Subsection \ref{subsec:practical-remarks}, we can handle more general right-hand sides $f$ through the alternative problem \eqref{eq:decomposition}.
The construction of $\phi_{\intpt}$ offers considerable flexibility, including the possibility of vertical translations, which in principle permits control of the boundary sign of the test function. However, it should be noted that since the inequalities in \eqref{eq:enclosure} are derived from \eqref{eq:u_s0_solution} by evaluating the boundary integral term from above and below by zero, the estimates become sharper when $\phi_{\intpt}$ is designed to minimize its boundary values.

While Corollary \ref{cor:enclosure} provides an estimate at a single point, combining this pointwise observation with the comparison principle established in Theorem \ref{thm:comparison_h10} allows us to extend these bounds to the entire domain. Applying this idea to the two-dimensional setting yields the following global enclosure result.

\begin{corollary}\label{cor:enclosure_comparison}
With the same notation and assumptions as in Corollary \ref{cor:enclosure},
consider two mappings
$\overline{\testFuncMap} : \intpt \mapsto \overline{\phi_{\intpt}}$ and
$\underline{\testFuncMap} : \intpt \mapsto \underline{\phi_{\intpt}}$.
For a Green-representable super-solution $\overline{u}$ with respect to $\overline{\testFuncMap}$ and
a Green-representable sub-solution $\underline{u}$ with respect to $\underline{\testFuncMap}$,
we have
\begin{equation}
\label{eq:enclosure_comparison}
\underline{u} \leq u \leq \overline{u} \quad \text{in } \Omega.
\end{equation}
\end{corollary}

Corollary \ref{cor:enclosure_comparison} provides a theoretical foundation for constructing global solution bands. The computational procedure for handling varying evaluation points $\intpt$ is described in Section \ref{subsec:singular_integrals}. However, a full numerical implementation is not conducted in this study. Instead, we focus on rigorous pointwise estimates via Corollary \ref{cor:enclosure}. It is worth emphasizing that obtaining guaranteed pointwise bounds for solutions with low regularity (e.g., in non-convex domains) has been a challenging problem in conventional methods such as finite element analysis. The proposed method achieves this naturally, providing a practical method for local verification. The extension to global bands remains an interesting direction for future work.

\subsection{Construction of Test Functions}\label{subsec:test_function_construction}
In this subsection, we describe a systematic approach for constructing test functions $\phi_{\intpt}$ of the form \eqref{eq:test-function} with \eqref{eq:harmonic-part}.
As indicated by Corollary \ref{cor:enclosure}, the sharpness of the rigorous enclosure depends directly on the magnitude of the boundary values of the test functions. Specifically, smaller deviations from zero on the boundary lead to tighter bounds for the solution $u(\intpt)$.
We begin by introducing a decomposition of $\phi_{\intpt}$ into a basic component $\phi_{\intpt}^0$ and a constant shift $C$:
\begin{equation}
\phi_{\intpt}^0 (x):= \intcoef\fundsol(\intpt,x) + \sum_{i=1}^n a_i\fundsol(s_i,x),
\end{equation}
where $\phi_{\intpt}:= \phi_{\intpt}^0 + C$.
This decomposition facilitates the optimization process described in the following algorithm.

\begin{algorithm}
\caption{Construction of Test Functions for Solution Enclosure}
\begin{algorithmic}[1]
\State Fix $\intpt$ and set $\intcoef=1$ 
\State Determine source points $\{s_1,\ldots,s_n\} \subset \mathbb{R}^2 \setminus \overline{\Omega}$ based on the domain geometry
\State Determine coefficients $\{a_1,\ldots,a_n\}$ to minimize $\|\phi_{\intpt}^0\|_{L^\infty(\partial\Omega)}$ (approximate homogeneous boundary condition)
\State Compute rigorous bounds for the boundary values: $m := \displaystyle\min_{x \in \partial\Omega}\phi_{\intpt}^0(x)$ and $M := \displaystyle\max_{x \in \partial\Omega}\phi_{\intpt}^0(x)$ using interval arithmetic
\State Construct test functions: $\overline{\phi_{\intpt}}=\phi_{\intpt}^0 - m$ (so that $\overline{\phi_{\intpt}} \ge 0$ on $\partial\Omega$) and \newline $\underline{\phi_{\intpt}}=\phi_{\intpt}^0 - M$ (so that $\underline{\phi_{\intpt}} \le 0$ on $\partial\Omega$)
\end{algorithmic}
\end{algorithm}
The key step in this construction is to approximate the harmonic function $v$ defined by
\begin{equation}
v = -\intcoef\Gamma_{\intpt} \quad \text{on } \partial\Omega, \quad \Delta v = 0 \quad \text{in } \Omega,
\end{equation}
so that $\phi_{\intpt}^0 = v + \intcoef\Gamma_{\intpt}$ nearly vanishes on the boundary.
To construct an approximate solution $\hat{v}$ using the fundamental solutions, we employ the MFS.
The implementation involves the following steps:
\begin{itemize}
\item Selection of source points $s_{1},s_{2},\dotsc,s_{n}\in\mathbb{R}^{2}\setminus\bar{\Omega}$.
\item Choice of collocation points $x_{1},x_{2},\dotsc,x_{n}\in\partial\Omega$.
\end{itemize}
Once the source and collocation points are fixed, the coefficients $a_i$ are determined by solving the linear system:
\begin{equation}
  \begin{pmatrix}\fundsol(s_{1},x_{1}) & \fundsol(s_{2},x_{1}) & \cdots & \fundsol(s_{n},x_{1})\\ \fundsol(s_{1},x_{2}) & \fundsol(s_{2},x_{2}) & \cdots & \fundsol(s_{n},x_{2})\\ \vdots & \vdots & \ddots & \vdots\\ \fundsol(s_{1},x_{n}) & \fundsol(s_{2},x_{n}) & \cdots & \fundsol(s_{n},x_{n})\end{pmatrix}\begin{pmatrix}a_{1}\\ a_{2}\\ \vdots\\ a_{n}\end{pmatrix} = \begin{pmatrix}-\fundsol(\intpt,x_{1})\\ -\fundsol(\intpt,x_{2})\\ \vdots\\ -\fundsol(\intpt,x_{n})\end{pmatrix}.
\end{equation}

It is worth noting that while theoretical convergence guarantees for the MFS typically require smoothness of $\partial\Omega$, the absence of such guarantees for domains does not compromise the rigor of our enclosure method.
In our framework, the MFS serves solely to generate a \textit{candidate} test function.
The validity of the final enclosure \eqref{eq:enclosure} relies exclusively on the rigorous evaluation of the boundary values $m$ and $M$ (Step 4 of the Algorithm), which is performed using interval arithmetic (as detailed in Section \ref{subsec:numericalexample}).
Even if the MFS approximation is inaccurate, the method remains mathematically rigorous, although the resulting bounds $[\underline{u}, \overline{u}]$ may become wider.
Therefore, empirical guidelines for source point placement, such as those established in \cite{amano1991error} for rectangular and L-shaped domains, can be effectively employed to obtain sharp enclosures.Our specific parameter choices for these domains are detailed in Section ~\ref{subsec:numericalexample}.

\subsection{Numerical Example}\label{subsec:numericalexample}

All numerical computations were performed using MATLAB R2024b with GNU C++ Compiler 11.4.0. To ensure the rigor of our numerical results, all rounding errors were strictly controlled using interval arithmetic implemented in the kv library version 0.4.57~\cite{kashiwagi2024kv}. 
The Schwarz--Christoffel Toolbox version 3.1.3~\cite{driscoll2024schwarz} was used in part of the source points computations.
We note that while our experiments were conducted on a workstation with four Intel Xeon Platinum 8380H processors (2.90 GHz) and 3 TB RAM running Ubuntu 22.04.4, the actual memory consumption was less than 1 GB. This indicates that the experiments can be reproduced on standard personal computers.

We conducted numerical experiments on the following boundary value problems:
\begin{equation}
    \begin{cases}
        -\Delta u = f_i & \text{in $\Omega$,}\\
        u = 0 & \text{on $\partial\Omega$,}
    \end{cases}
\end{equation}
where three test cases were considered.
\begin{itemize}
    \item $f_1(x,y) = 1$ (constant source term)
    \item $f_2(x,y)=(x-0.125)^2+(y-0.25)^3$ (polynomial source term)
    \item $f_3(x,y) = x + \sin((x+0.5)y^2)$ (non-polynomial variable source term)
\end{itemize}

Let $\Omega_{\square}$ denote the square domain $(-0.5,0.5)^2$ and $\Omega_{\mathsf{L}}$ denote the L-shaped domain $(-1,1)^2\setminus[0,1]^2$.
For both domains, we employed 69 collocation points and 69 source points as illustrated in Fig.~\ref{fig:charge_sim_overview}. The numerical results are presented in Tables \ref{tab:square_results} and \ref{tab:lshape_results}, where the relative error is computed as the ratio of the interval width to the midpoint of the interval. Note that the interval widths are computed from the original interval endpoints before rounding; therefore, they may differ slightly from the differences of the displayed rounded values. The computed solution enclosures are visualized in Fig.~\ref{fig:domain_results_overview}.
These results demonstrate that the proposed method achieves highly accurate enclosures for the square domain. At the origin $(0,0)$, the interval widths are on the order of $10^{-7}$ to $10^{-8}$, corresponding to relative errors of $10^{-6}$ to $10^{-4}$, while at the evaluation point $(0.25,0.25)$, the interval widths are on the order of $10^{-6}$. 

For the L-shaped domain, which presents additional challenges due to the corner singularity at the origin, the enclosures are significantly wider. The interval widths are on the order of $10^{-3}$ for the constant and polynomial source terms ($f_1$ and $f_2$), corresponding to relative errors of approximately $5$--$7\%$, and on the order of $10^{-2}$ for the non-polynomial variable source term $f_3$, with relative errors reaching up to $59\%$. This degradation in precision is expected. The corner singularity reduces solution regularity, and the MFS approximation becomes less effective near the reentrant corner, resulting in test functions whose boundary values deviate more significantly from zero.
Despite the wider enclosures obtained for the L-shaped domain, these results are significant because they provide \emph{rigorous} pointwise bounds without assuming $W^{2,p}$ regularity---a setting in which conventional finite element error estimates are not applicable. Although varying $\intpt$ continuously throughout the entire domain is not pursued in this paper, as it would require recomputation of the MFS approximation for each evaluation point, the present results represent, to the best of our knowledge, the first instance of guaranteed pointwise estimates for the Poisson equation in non-convex domains.

Notably, the polynomial source term $f_2$ exhibits interval widths comparable to or slightly smaller than the constant source term $f_1$, while the non-polynomial variable source term $f_3$ containing trigonometric functions shows larger interval widths. This is expected due to the additional complexity in evaluating integrals involving trigonometric functions with interval arithmetic.
The precision of the obtained enclosures, as shown in Tables \ref{tab:square_results} and \ref{tab:lshape_results}, validates our theoretical framework for both convex and non-convex domains. Moreover, the results indirectly demonstrate that the chosen MFS configuration effectively produced test functions that are sufficiently close to zero on the boundary---particularly for the square domain---which is essential for achieving sharp bounds.

% ============================================================
% Figures and Tables
% ============================================================

\begin{figure}[htbp]
    \begin{minipage}{0.5\linewidth}
        \centering
        \includegraphics[width=7cm]{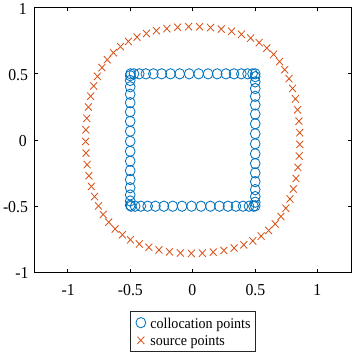}
        \subcaption{Square domain $\Omega_{\square}$}
        \label{fig:rect_chargesim_1}
    \end{minipage}%
    \begin{minipage}{0.5\linewidth}
        \centering
        \includegraphics[width=7cm]{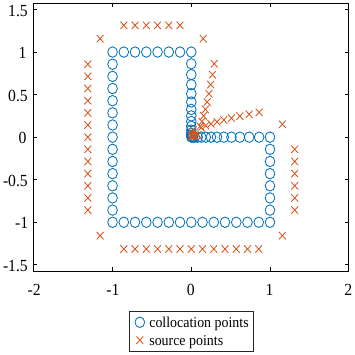}
        \subcaption{L-shaped domain $\Omega_{\mathsf{L}}$}
        \label{fig:lshape_chargesim_1}
    \end{minipage}
    \caption{Distribution of collocation and source points for the test domains. Red dots indicate collocation points. Both domains contain 69 points.}
    \label{fig:charge_sim_overview}
\end{figure}

\begin{figure}[htbp]
    \noindent
    \begin{minipage}{0.5\linewidth}
        \centering
        \includegraphics[width=7cm]{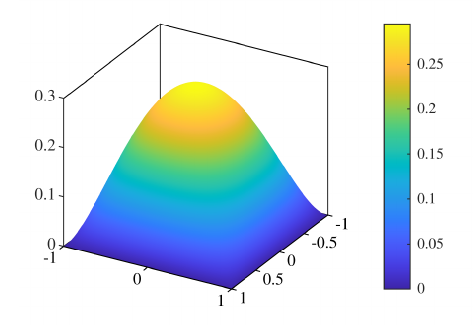}
        \subcaption{$\Omega_{\square}$ \& $f_{1}$}
        \label{fig:rect_chargesim_2}
    \end{minipage}%
    \begin{minipage}{0.5\linewidth}
        \centering
        \includegraphics[width=7cm]{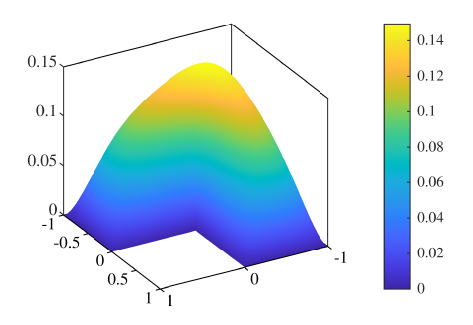}
        \subcaption{$\Omega_{\mathsf{L}}$ \& $f_{1}$}
        \label{fig:lshape_chargesim_2}
    \end{minipage}%

    \bigskip
    \noindent
    \begin{minipage}{0.5\linewidth}
        \centering
        \includegraphics[width=7cm]{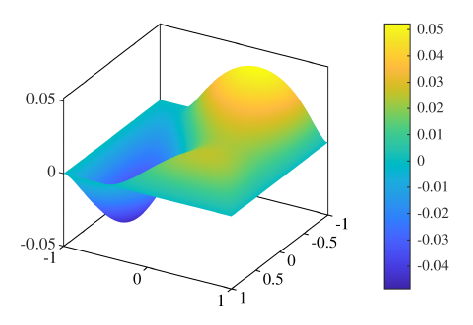}
        \subcaption{$\Omega_{\square}$ \& $f_{2}$}
        \label{fig:rect_chargesim_3}
    \end{minipage}%
    \begin{minipage}{0.5\linewidth}
        \centering
        \includegraphics[width=7cm]{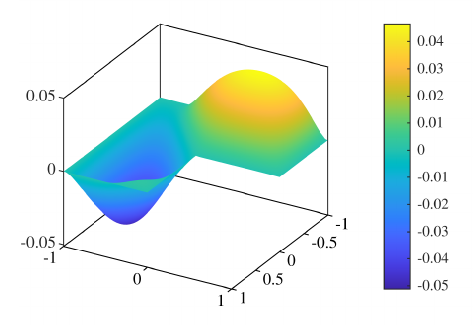}
        \subcaption{$\Omega_{\mathsf{L}}$ \& $f_{2}$}
        \label{fig:lshape_chargesim_3}
    \end{minipage}

    \bigskip
    \noindent
    \begin{minipage}{0.5\linewidth}
        \centering
        \includegraphics[width=7cm]{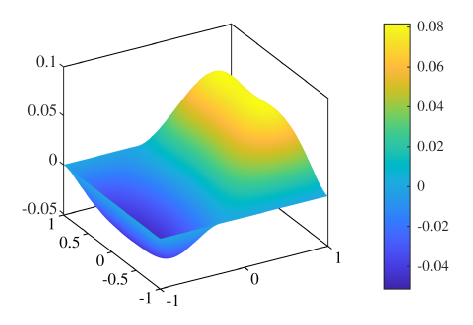}
        \subcaption{$\Omega_{\square}$ \& $f_{3}$}
        \label{fig:rect_chargesim_4}
    \end{minipage}%
    \begin{minipage}{0.5\linewidth}
        \centering
        \includegraphics[width=7cm]{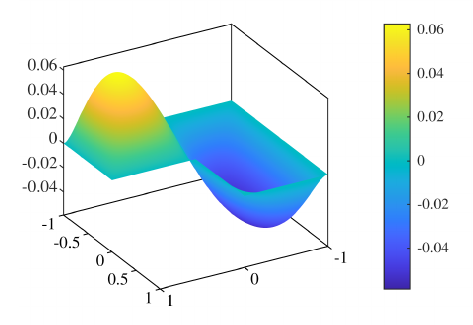}
        \subcaption{$\Omega_{\mathsf{L}}$ \& $f_{3}$}
        \label{fig:lshape_chargesim_4}
    \end{minipage}
    \caption{Computed solution enclosures for the square domain $\Omega_{\square}$ (left column) and the L-shaped domain $\Omega_{\mathsf{L}}$ (right column) with source terms $f_1$, $f_2$, and $f_3$ (top to bottom). Each plot shows the upper and lower bounds of the enclosure.}
    \label{fig:domain_results_overview}
\end{figure}

\begin{table}[htbp]
\centering
\caption{Enclosure results for the square domain $\Omega_{\square}$.}
\label{tab:square_results}
\small
\begin{tabular}{ccccc}
\hline
Evaluation Point $\intpt$ & Source Term $f$ & Enclosure Interval & Interval Width & Relative Error \\
\hline
\multirow{3}{*}{$(0,0)$} & $f_1$ & $[7.367\times 10^{-2},7.368\times 10^{-2}]$ & $2.92\times 10^{-7}$ & $3.96\times 10^{-6}$ \\
& $f_2$& $[1.134\times 10^{-2},1.135\times 10^{-2}]$ & $7.15\times 10^{-8}$ & $6.31\times 10^{-6}$ \\
& $f_3$& $[1.395\times 10^{-3},1.396\times 10^{-3}]$ & $3.04\times 10^{-7}$ & $2.18\times 10^{-4}$ \\
\hline
\multirow{3}{*}{$(0.25,0.25)$} & $f_1$ & $[4.528\times 10^{-2},4.529\times 10^{-2}]$ & $5.94\times 10^{-6}$ & $1.32\times 10^{-4}$ \\
& $f_2$& $[4.049\times 10^{-3},4.052\times 10^{-3}]$ & $1.46\times 10^{-6}$ & $3.59\times 10^{-4}$ \\
& $f_3$& $[7.702\times 10^{-3},7.709\times 10^{-3}]$ & $6.19\times 10^{-6}$ & $8.03\times 10^{-4}$ \\
\hline
\end{tabular}
\end{table}

\begin{table}[htbp]
\centering
\caption{Enclosure results for the L-shaped domain $\Omega_{\mathsf{L}}$.}
\label{tab:lshape_results}
\small
\begin{tabular}{ccccc}
\hline
Evaluation Point $\intpt$ & Source Term $f$ & Enclosure Interval & Interval Width & Relative Error \\
\hline
\multirow{3}{*}{$(-0.5,-0.5)$} & $f_1$ & $[1.234\times 10^{-1},1.330\times 10^{-1}]$ & $9.53\times 10^{-3}$ & $7.44\times 10^{-2}$ \\
& $f_2$ & $[1.067\times 10^{-1},1.151\times 10^{-1}]$ & $8.29\times 10^{-3}$ & $7.47\times 10^{-2}$ \\
& $f_3$ & $[-6.05\times 10^{-2},-3.28\times 10^{-2}]$ & $2.76\times 10^{-2}$ & $5.92\times 10^{-1}$ \\
\hline
\multirow{3}{*}{$(0.5,-0.5)$} & $f_1$ & $[9.855\times 10^{-2},1.034\times 10^{-1}]$ & $4.82\times 10^{-3}$ & $4.77\times 10^{-2}$ \\
& $f_2$ & $[7.417\times 10^{-2},7.836\times 10^{-2}]$ & $4.19\times 10^{-3}$ & $5.49\times 10^{-2}$ \\
& $f_3$ & $[5.199\times 10^{-2},6.593\times 10^{-2}]$ & $1.40\times 10^{-2}$ & $2.37\times 10^{-1}$ \\
\hline
\end{tabular}
\end{table}

\subsubsection*{Distribution of Singular Points}
Here, we describe the method used to determine the locations of singular points and collocation points in our numerical implementation.

First, let us explain the procedure for the square domain $\Omega_{\square}$. For convenience, we identify $\mathbb{R}^{2}$ with $\mathbb{C}$. Since $\Omega_{\square}$ is a simple polygon, there exists a conformal mapping $\Phi(z)$ from the unit disk $D=\lbrace z\in\mathbb{C}:\lvert z\rvert<1\rbrace$ to the exterior of $\Omega_{\square}$, which can be constructed using the Schwarz--Christoffel mapping. Although simpler point distribution strategies can achieve reasonable accuracy for the square domain, we employed this conformal mapping approach with the expectation of obtaining slightly better numerical results. The conformal mapping $\Psi(z)=\Phi(1/z)$ maps the exterior of $D$ to the exterior of $\Omega_{\square}$. Using this mapping, we define the collocation points $x_{1},x_{2},\dotsc,x_{n}$ and source points $s_{1},s_{2},\dotsc,s_{n}$ through the following equations
\begin{equation}
    x_{k} = \Psi(\mathrm{e}^{2\pi\mathrm{i}k/n}),
    \quad s_{k} = \Psi\biggl(\frac{3}{2}\mathrm{e}^{2\pi\mathrm{i}k/n}\biggr)\quad(k=1,2,\dotsc,n).
\end{equation}
For the L-shaped domain $\Omega_{\mathsf{L}}$, as reported in \cite{amano1991error}, placing source points and collocation points densely near the reentrant corner can reduce the boundary error. Accordingly, we placed collocation points with higher density near the reentrant corner and determined the source points based on Amano's arrangement~\cite{amano1991error}
\begin{equation}
    s_{k} = x_{k}-\frac{\mathrm{i}r_{k}}{2}(x_{k+1}-x_{k-1}),
\end{equation}
where we arrange $x_{1},x_{2},\dotsc,x_{n}$ so that their arguments are monotonically increasing and define $x_{0}\coloneq x_{n}$ and $x_{n+1}\coloneq x_{1}$. In the standard Amano arrangement, $r_{k}$ is a positive constant independent of $k$, in our experiments, however, we modified $r_{k}$ as follows:
\begin{equation*}
    r_{k} = \frac{R_{k}-1}{\sin(2\pi/n)},
    \quad\text{where}~R_{k} = \begin{cases}1.05 & \text{if $x_{k}\in[0,0.1]^{2}$,}\\ 1.2 & \text{if $x_{k}\notin[0,0.1]^{2}$}\end{cases}
\end{equation*}
Note that the formula $r=(R-1)/\sin(2\pi/n)$ is based on \cite[Theorem 4.1]{sakakibara2020bidirectional}.

\subsection{Rigorous Singular Integrals with Logarithmic Potentials}\label{subsec:singular_integrals}
After determining $\phi_{\intpt}$ using the method described in Subsection \ref{subsec:test_function_construction}, applying Corollary \ref{cor:enclosure} requires the rigorous evaluation of the integral
\begin{equation}
\int_{\Omega}f\cdot\phi_{\intpt}\,dx.
\end{equation}
The fundamental integral calculation underlying this evaluation is
\begin{equation}
   I = \iint_{\Omega}f(x,y)\log((x-x_{0})^{2}+(y-y_{0})^{2})\,dx\,dy,
\end{equation}
where $\intpt=(x_{0},y_{0})\in\bar{\Omega}$. 
While our original framework only requires $\intpt$ to be an interior point of $\Omega$, from the perspective of integration theory, $\intpt$ may also lie on the boundary provided that $f$ possesses sufficient regularity for $I$ to be integrable. This extension of the assumptions is particularly useful when we need to consider $\intpt$ approaching the boundary, as in applications of Corollary \ref{cor:enclosure_comparison}.
In what follows, we present our approach for obtaining rigorous enclosures of the integral value near $\intpt$. We assume that $\Omega$ is a bounded polygonal domain and that $f$ is sufficiently smooth (smooth enough to admit Taylor expansions, though this condition could potentially be relaxed to include piecewise smooth functions that ensure integrability of $I$).

We begin by constructing a triangulation of $\Omega$ such that $\intpt$ is located at one of the vertices; for example, when $\Omega$ is rectangular, the desired triangulation can be obtained as illustrated in Figure \ref{fig:triangulation}.
For simplicity of exposition, we assume $x_0=y_0=0$ and that $\Omega$ is a triangle defined by $\{(x,y)\in\mathbb{R}^2:0<ay<bx<ab\}$. Under the variable transformation $u=x$, $k=y/x$, the integral $I_0$ over this triangle can be expressed as
\begin{equation}
   I_0 = \underbrace{\int_{0}^{a}\biggl(2u\log(u)\int_{0}^{b/a}f(u,ku)\,dk\biggr)\,du}_{I_1} 
   + 
   \underbrace{\int_{0}^{a}\!\!\int_{0}^{b/a}uf(u,ku)\log(1+k^{2})\,dk\,du}_{I_2}.
\end{equation}
We denote the first term on the right-hand side by $I_1$ and the second term as $I_2$. The integrand of $I_2$ does not exhibit singularities in the integration domain, making it comparatively easier to compute than $I_1$. However, since we seek rigorous enclosures rather than approximations for $I_2$, conventional numerical methods are insufficient. We therefore employ the kv library \cite{kashiwagi2024kv}, which is based on power series arithmetic and interval arithmetic, to accomplish this. The same approach applies to integrals over triangles that do not have $\intpt$ as a vertex.
We now focus on explaining the enclosure method for $I_1$. We first compute bounded closed intervals $[c_{i,j}]=[\underline{c}_{i,j},\overline{c}_{i,j}]\subset\mathbb{R}$ ($0\leq i\leq m$, $0\leq j\leq n$) such that
\begin{equation}
   f(u,ku) \in \sum_{i=0}^{m}\sum_{j=0}^{n}[c_{i,j}]k^iu^j\quad\text{for all $(u,k)\in[0,a]\times[0,b/a]$.}
   \label{eq:psa}
\end{equation}
The intervals $[c_{i,j}]$ are computed using Power Series Arithmetic \cite{kashiwagi2019seido,tanaka2022rigorous}, incorporating interval arithmetic to account for all rounding errors. From equation \eqref{eq:psa}, we obtain the following enclosure for $I_1$:
\begin{equation*}
   I_1 \in \sum_{i=0}^{m}\frac{(b/a)^{i+1}}{i+1}\sum_{j=0}^{n}[c_{i,j}]\int_0^au^{j+1}\log(u)\,du.
\end{equation*}
The integral on the right-hand side can be evaluated using the formula
\begin{equation}
   \int_0^au^{j+1}\log(u)\,du = \frac{a^{j+2}((j+2)\log(a)-1)}{(j+2)^2}.
\end{equation}
Based on this formula, we can obtain rigorous bounds for $I$ through sequential calculations using interval arithmetic, considering all numerical computation errors.
\begin{figure}[tbp]
    \centering
    \begin{tikzpicture}[scale=0.475]
        \coordinate (A) at (-4,-3);
        \coordinate (B) at (-4,7);
        \coordinate (C) at (7,7);
        \coordinate (D) at (7,-3);
        \fill [lightgray] (0,0)--(7,0)--(7,7)--cycle;
        \draw (A)--(B)--(C)--(D)--cycle;
        \draw [dashed] (A)--(0,0);
        \draw [dashed] (B)--(0,0);
        \draw [dashed] (C)--(0,0);
        \draw [dashed] (D)--(0,0);
        \draw [dashed] (-4,0)--(7,0);
        \draw [dashed] (0,-3)--(0,7);
        \node [above left] at (0,-0.2) {$\intpt$};
        \node [below] at (3.5,0) {$a$};
        \node [left] at (7,3.5) {$b$};
    \end{tikzpicture}
    \caption{Triangulation of a rectangular domain. The gray triangle illustrates a subdomain where the singular integral is evaluated.}
    \label{fig:triangulation}
\end{figure}
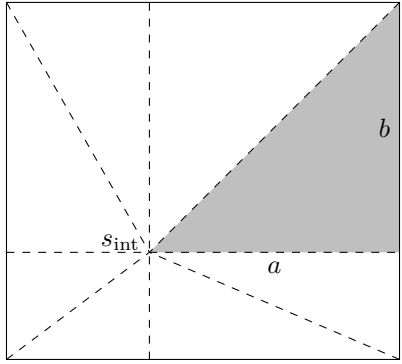
The numerical method described above assumes a fixed $\intpt$. This procedure is sufficient for computing pointwise upper and lower bounds of $u$ using Corollary \ref{cor:enclosure}. 
However, applying Corollary \ref{cor:enclosure_comparison} requires evaluation of the integral $\int_{\Omega}f\cdot\phi_{\intpt}\,dx$ for all elements of the families $\{\overline{\phi_{\intpt}}\}_{\intpt \in \Omega}$ and $\{\underline{\phi_{\intpt}}\}_{\intpt \in \Omega}$ that constitute our super- and sub-solutions (where we use $\phi_{\intpt}$ to denote either $\overline{\phi_{\intpt}}$ or $\underline{\phi_{\intpt}}$ in the integral). This requirement entails handling integrals with $\intpt$ across the entire domain, which is not pursued in this paper but can be addressed using the approach outlined below.
In principle, obtain solution enclosures over the whole domain $\Omega$ using Corollary \ref{cor:enclosure_comparison} by the following procedure.
\begin{enumerate}
\item Partitioning $\Omega$ into small triangular or rectangular meshes
\item Within each mesh element:
   \begin{itemize}
   \item Fixing all parameters except $\intpt$ (including exterior source points and their coefficients);
   \item Treating $\intpt$ as an interval variable.
   \item Evaluating the integral $I$ using the method described above
   \end{itemize}
\end{enumerate}
The techniques described above provide a complete framework for rigorous integral evaluation. In this paper, we limit our numerical experiments to pointwise estimates, for which these methods are directly applicable and computationally tractable. Although the construction of global solution bands (Corollary \ref{cor:enclosure_comparison}) is excluded from the present numerical results due to the complexity associated with varying $\intpt$ continuously, the established pointwise capability itself represents a significant advancement in verified computing for singular problems. Extending the implementation to global bands constitutes a natural next step in the development of our methodology.

\section{Conclusion and Future Work}\label{sec:conclusion}
In this paper, we propose a novel framework for super- and sub-solutions using test functions based on fundamental solutions. By introducing the concept of ``Green-representable solutions'' and elucidating their theoretical properties, we formulated super- and sub-solutions for a broader class of functions, including piecewise linear functions, which are difficult to handle using conventional definitions based on variational inequalities. This framework has the advantage of naturally yielding rigorous pointwise estimates even for solutions that do not possess $H^2$ regularity, a setting that often presents challenges in conventional finite element analysis. Through numerical experiments, we demonstrate that the proposed method provides rigorous and high-precision enclosures for solutions to problems with constant and discontinuous source terms in one-dimensional and two-dimensional polygonal domains (including non-convex domains).
A secondary but important contribution of this study is the identification of a new connection between verified computing and the MFS. In our framework, the sharpness of the solution enclosure depends on the behavior of the test function on the boundary, that is, how accurately the test function can approximate the boundary conditions. Therefore, the quality of the obtained error estimate also serves as an indirect and quantitative indicator of the approximation performance of the MFS. This observation implies that improvements in the approximation accuracy of the MFS directly translate into improved accuracy in rigorous solution verification.

Future work includes the following points.
First, the construction of global solution enclosures in domains of two or more dimensions. In this paper, the numerical experiments for the 2D case focused on pointwise estimates at fixed evaluation points $\intpt$. To construct a band of super- and sub-solutions covering the entire domain, as achieved in the 1D case, it is necessary to control the parameters of the test functions continuously or densely in response to changes in the evaluation point $\intpt$. This task involves challenges such as the rigorous evaluation of parameter interpolation errors and the reduction of computational cost, but it represents an essential step toward enhancing the practical applicability of the proposed method.
Second, the optimization of test function construction. In this study, we relied on existing heuristics and standard linear solvers for the placement of MFS source points and the determination of coefficients. However, there is considerable scope for optimizing the parameters of the test functions (especially the positions of external source points) with respect to the objective of minimizing boundary values while satisfying the prescribed sign conditions on the boundary ($\phi_{\intpt} \ge 0$ or $\le 0$). For example, exploring more effective source point configurations and coefficient selection strategies using gradient-based optimization or machine learning approaches (such as unsupervised learning using optimizers like Adam) could substantially improve the sharpness of the resulting enclosures.
Third, the extension to more general equations. Although the present framework relies heavily on the properties of fundamental solutions, similar ideas may be applicable to other linear elliptic equations for which fundamental solutions are available (e.g., the Helmholtz equation or advection-diffusion equations).
Addressing these directions is expected to further expand both the theoretical scope and practical applicability of verified computing frameworks based on fundamental solutions and to contribute to the development of highly reliable numerical simulation technologies.

\section*{Acknowledgments}
This work was supported by the JST FOREST Program (Grant Number JPMJFR202S, Japan).
This work was also supported in part by participation in the JST Open Problems Workshop in Mathematical Sciences 2023. 
We thank Koya Sakakibara at Kanazawa University for his valuable advice. We would like to thank Editage (www.editage.jp) for English language editing. 

\bibliographystyle{plain}
\bibliography{ref}

\appendix
\section{Technical Lemma for Trace Estimates} \label{appen:proof_of_lemma_3_5}
\begin{lemma}\label{lem:trace_value_general_n_dim}
Let $N \geq 2$ be an integer and $1\leq p,q\leq\infty$ be fixed. Let $0<\lambda<1$ and let $\Omega$ be a bounded Lipschitz domain in $\mathbb{R}^N$. Define the scaled domain as $\Omega_{\lambda} := \{\lambda x : x \in \Omega\}$. Assume that there exists a constant $C(\Omega,p,q)>0$ such that
\[
\left\|u\right\|_{L^p(\partial \Omega)} \leq C(\Omega,p,q)\|u\|_{W^{1,q}(\Omega)}.
\]
Then, for any $v\in W^{1,q}(\Omega_{\lambda})$, we have
\begin{equation}
\left\|v\right\|_{L^{p}(\partial\Omega_{\lambda})}\leq C(\Omega,p,q)\lambda^{\frac{N-1}{p}  - \frac{N}{q}}\left\|v\right\|_{W^{1,q}(\Omega_{\lambda})},
\end{equation}
where we interpret $\lambda^{\frac{N-1}{p} - \frac{N}{q}}=\lambda^{-\frac{N}{q}}$ when $p=\infty$, $\lambda^{\frac{N-1}{p} - \frac{N}{q}}=\lambda^{\frac{N-1}{p}}$ when $q=\infty$, and $\lambda^{\frac{N-1}{p} - \frac{N}{q}}=1$ when $p=q=\infty$.
\end{lemma}
\begin{proof}
Let $y = \lambda x$ and define $u(x) := v(\lambda x)$ for $x \in \Omega$.

First, we consider the scaling of the function value in $L^p$ norm. For $1\leq p<\infty$, using the change of variables (noting that the boundary measure scales as $dS_y = \lambda^{N-1}dS_x$), we have:
\[
\begin{aligned}
\left\|v\right\|_{L^p(\partial \Omega_\lambda)}^p
&= \int_{\partial\Omega_\lambda} \left|v(y)\right|^p\,dS_y \\
&= \int_{\partial\Omega} \left|v(\lambda x)\right|^p\,\lambda^{N-1}\,dS_x \\
&= \int_{\partial\Omega} \left|u(x)\right|^p\,\lambda^{N-1}\,dS_x \\
&= \lambda^{N-1}\left\|u\right\|_{L^p(\partial \Omega)}^p.
\end{aligned}
\]
Therefore,
\[
\left\|v\right\|_{L^p(\partial \Omega_\lambda)} = \lambda^{(N-1)/p} \left\|u\right\|_{L^p(\partial \Omega)}.
\]

For $p=\infty$, the boundary measure scaling does not affect the supremum norm:
\[
\left\|v\right\|_{L^\infty(\partial \Omega_\lambda)} = \left\|u\right\|_{L^\infty(\partial \Omega)}.
\]

Now, for the $W^{1,q}$ norm when $1\leq q<\infty$, we have:
\[
\begin{aligned}
\|v\|_{W^{1,q}(\Omega_\lambda)}^q &= \lambda^N \|u\|_{L^q(\Omega)}^q + \lambda^{N-q} \|\nabla u\|_{L^q(\Omega)}^q.
\end{aligned}
\]
Since $0 < \lambda < 1$ and $q \geq 1$, we have $\lambda^{N-q} \geq \lambda^N$, and thus
\[
\|v\|_{W^{1,q}(\Omega_\lambda)}^q \geq \lambda^N (\|u\|_{L^q(\Omega)}^q + \|\nabla u\|_{L^q(\Omega)}^q) = \lambda^N \|u\|_{W^{1,q}(\Omega)}^q.
\]
Taking the $q$-th root,
\[
\|v\|_{W^{1,q}(\Omega_\lambda)} \geq \lambda^{\frac{N}{q}}\|u\|_{W^{1,q}(\Omega)},
\]
which gives the upper bound we need:
\[
\|u\|_{W^{1,q}(\Omega)} \leq \lambda^{-\frac{N}{q}}\|v\|_{W^{1,q}(\Omega_\lambda)}.
\]
Combining these inequalities with the assumption yields
\[
\begin{aligned}
\left\|v\right\|_{L^p(\partial \Omega_\lambda)} &= \lambda^{\frac{N-1}{p}} \left\|u\right\|_{L^p(\partial \Omega)} \\
&\leq \lambda^{\frac{N-1}{p}} C(\Omega,p,q)\|u\|_{W^{1,q}(\Omega)} \\
&\leq C(\Omega,p,q)\lambda^{\frac{N-1}{p} -\frac{N}{q}}\|v\|_{W^{1,q}(\Omega_\lambda)}.
\end{aligned}
\]

For $q=\infty$, we have $\|u\|_{L^\infty(\Omega)} = \|v\|_{L^\infty(\Omega_\lambda)}$ and $\|\nabla u\|_{L^\infty(\Omega)} = \lambda \|\nabla v\|_{L^\infty(\Omega_\lambda)} \leq \|\nabla v\|_{L^\infty(\Omega_\lambda)}$, so
\[
\|u\|_{W^{1,\infty}(\Omega)} = \max\{\|u\|_{L^\infty(\Omega)}, \|\nabla u\|_{L^\infty(\Omega)}\} \leq \max\{\|v\|_{L^\infty(\Omega_\lambda)}, \|\nabla v\|_{L^\infty(\Omega_\lambda)}\} = \|v\|_{W^{1,\infty}(\Omega_\lambda)}.
\]
Combined with the trace scaling, this yields
\[
\left\|v\right\|_{L^p(\partial \Omega_\lambda)} \leq \lambda^{\frac{N-1}{p}} C(\Omega,p,\infty)\|v\|_{W^{1,\infty}(\Omega_\lambda)}.
\]
\end{proof}

\end{document}